\documentclass[12pt]{article}

\usepackage{amsmath, amssymb, amscd, amsthm}

\usepackage[pdftex, colorlinks=true,  citecolor=blue, linkcolor=blue, linktocpage=true]{hyperref}

\usepackage[all]{xy}

\theoremstyle{plain}
\newtheorem{thm}{Theorem}

\newtheorem{theorem}{Theorem}[section]
\newtheorem{lemma}[theorem]{Lemma}
\newtheorem{proposition}[theorem]{Proposition}
\newtheorem{corollary}[theorem]{Corollary}

\theoremstyle{definition}

\newtheorem{remark}[theorem]{Remark}

\newtheorem{example}[theorem]{Example}

\newcommand\bA{{\mathbb A}}

\newcommand\bG{{\mathbb G}}

\newcommand\bP{{\mathbb P}}
\newcommand\bQ{{\mathbb Q}}

\newcommand\bV{{\mathbb V}}

\newcommand\bZ{{\mathbb Z}}

\newcommand\bfmu{\boldsymbol{\mu}}

\newcommand\cA{{\mathcal A}}
\newcommand\cB{{\mathcal B}}
\newcommand\cC{{\mathcal C}}

\newcommand\cF{{\mathcal F}}

\newcommand\cH{{\mathcal H}}
\newcommand\cI{{\mathcal I}}
\newcommand\cJ{{\mathcal J}}

\newcommand\cL{{\mathcal L}}

\newcommand\cO{{\mathcal O}}

\newcommand\cT{{\mathcal T}}

\newcommand\cHom{{\cH}{\rm om}}

\newcommand\fm{\mathfrak{m}}
\newcommand\fn{\mathfrak{n}}

\newcommand\ch{{\rm ch}}

\renewcommand\deg{{\rm deg}}

\renewcommand\div{{\rm div}}

\newcommand\fr{{\rm fr}}
\newcommand\gp{{\rm gp}}

\newcommand\mult{{\rm mult}}

\newcommand\pr{{\rm pr}}

\newcommand\reg{{\rm reg}}

\newcommand\tr{{\rm tr}}

\newcommand\Aut{{\rm Aut}}

\DeclareMathOperator\Cl{{\rm Cl}}

\newcommand\Der{{\rm Der}}

\newcommand\GL{{\rm GL}}

\DeclareMathOperator\Hom{Hom}

\newcommand\Lie{{\rm Lie}}

\newcommand\Pic{{\rm Pic}}

\newcommand\Spec{{\rm Spec}}
\DeclareMathOperator\Stab{Stab}
\DeclareMathOperator\Supp{Supp}
\newcommand\Sym{{\rm Sym}}

\numberwithin{equation}{section}

\title{Equivariantly normal varieties \\
for diagonalizable group actions}

\author{Michel Brion}

\date{}

\begin{document}

\maketitle

\begin{abstract}
Given a finite group scheme $G$ over a field and 
a $G$-variety $X$, we obtain a criterion for $X$ 
to be $G$-normal in the sense of \cite{Br24}.  
When $G$ is diagonalizable, we describe the local
structure of $G$-normal varieties in codimension 
$1$ and their dualizing sheaf. As an application, 
we obtain a version of the Hurwitz formula for 
$G$-normal varieties, where $G$ is linearly 
reductive.
\end{abstract}

\section{Introduction}
\label{sec:int}

This paper is a sequel to \cite{Br24},
which considered actions of finite group 
schemes over fields of positive characteristic. 
In this setting, an action on a variety
$X$ does not necessarily lift to an
action on the normalization $\tilde{X}$; 
as a remedy, the notion of equivariant
normalization was introduced and explored. 
The present paper focuses on $G$-normal 
varieties, where $G$ is a finite 
diagonalizable group scheme, 
with two main motivations in mind: 
firstly, finite diagonalizable group
schemes have an especially simple structure 
and representation theory, which makes them 
very accessible. Secondly, they are quite
ubiquitous; for example, the study of
$G$-normal varieties with $G$ linearly reductive
reduces to some extent to the diagonalizable 
case (see Remark \ref{rem:linred} for details,
and Section \ref{sec:lrgs} for an illustration).

Given a finite diagonalizable group scheme
$G$ acting faithfully on a $G$-normal 
variety $X$ with categorical quotient 
$\pi : X \to Y$, 
the variety $Y = X/G$ is known to be normal 
and $\pi$ is a $G$-torsor over a dense open 
subset of $Y$. We may thus view $X$ as a
``ramified $G$-cover'' of $Y$ 
(as a partial converse,
every variety obtained as a $G$-torsor
over a normal variety is $G$-normal). 
When $G$ is constant, $\pi$ is indeed 
a ramified Galois cover with group $G$
and $X$ is normal; this gives back
the classically studied abelian covers, 
see e.g.~\cite{Pardini, AP}. But the general
case has quite different geometric features: 
if $G$ is infinitesimal, then $\pi$ is 
purely inseparable,
and hence ramified everywhere. Moreover,
$X$ is generally singular in codimension $1$,
with cuspidal singularities only. Still, 
$X$ satisfies Serre's property $(S_2)$, 
as well as an equivariant version of $(R_1)$: 
every $G$-stable divisor is Cartier at its
generic points. 

For simplicity, we now state our main results 
in the setting of curves over an algebraically
closed field $k$. Let $X$ be such a curve,
equipped with a faithful $G$-action with 
quotient $\pi : X \to Y$. We assume that
$X$ is $G$-normal; this is equivalent to
the orbit $G \cdot x \subset X$ 
being a Cartier divisor for any $x \in X(k)$,
and implies that the curve $Y$ is smooth. 
Denote by $\Stab_G(x) \subset G$ the stabilizer 
of $x$; it acts linearly on the fibre at $x$ 
of the conormal sheaf to $G \cdot x$ 
via a character $\nu(x) : \Stab_G(x) \to \bG_m$.
We may now describe the equivariant local 
structure of $X$ at $x$:

\begin{thm}\label{thm:local}

\begin{enumerate}

\item[{\rm (i)}] The group scheme
$H = \Stab_G(x)$ is cyclic, and $\nu(x)$
generates its character group.

\item[{\rm (ii)}] There exists
an open $G$-stable neighborhood 
$U = U(x) \subset X$ of $x$
such that the quotient morphism
$U \to U/G = V$ factors as 
$U \stackrel{\varphi}{\longrightarrow}
U/H \stackrel{\psi}{\longrightarrow} V$,
where $\varphi$ is a cyclic cover 
of degree $\vert H \vert$, 
and $\psi$ is a $G/H$-torsor.

\end{enumerate}

\end{thm}

By a cyclic cover of degree $n$ in this 
local setting, we mean a morphism of affine
schemes $\varphi: \Spec(A) \to \Spec(B)$, 
where $A = B[T]/(T^n -g)$ and $g \in B$
is a nonzerodivisor. Then $\varphi$
is the quotient by the cyclic group scheme
$\bfmu_n$ acting on $\Spec(A)$ via the 
$\bZ/n\bZ$-grading 
$A = \bigoplus_{m=0}^{n-1} B \bar{T}^m$; 
also, $\varphi$ is a $\bfmu_n$-torsor outside 
of the zero scheme of $g$. 

As a consequence of Theorem 
\ref{thm:local}, every $G$-normal curve
is a tamely ramified $G$-torsor in
the sense of \cite[Def.~2.1]{BB}.
But equivariant normality imposes
additional conditions; for example,
a cyclic cover as above is 
$\bfmu_n$-normal if and only if
$B$ is an integrally closed domain
and the zero scheme of $g$ is reduced.

As a further consequence, the finite 
morphism $\pi$ is flat and 
a local complete intersection. So
its dualizing sheaf is isomorphic to
the relative canonical sheaf 
$\omega_{X/Y}$, and hence is invertible.

\begin{thm}\label{thm:hurwitz}
The sheaf $\omega_{X/Y}$ is equipped with
a $G$-linearization and a $G$-invariant
global section $s_{X/Y}$ such that
\[ \div(s_{X/Y})  = 
\sum \big( \vert \Stab_G(x) \vert - 1)\big) \, G \cdot x  \]
(sum over the $G$-orbits of $k$-rational
points of $X$).
\end{thm}

For a projective $G$-normal curve $X$,
this readily yields the relation between
arithmetic genera
\begin{equation}\label{eqn:hur}
2 p_a(X) - 2 = \vert G \vert \, 
\big( 2 p_a(Y) -2 \big) 
+ \sum \big(  \vert \Stab_G(x) \vert -1 \big), 
\end{equation}
where the sum runs again over the orbits.
(We have of course $p_a(Y) = g(Y)$ as $Y$ 
is smooth; also, $g(X) = g(Y)$ if $G$ is 
infinitesimal, see \cite[Rem.~5.2]{Br24}). 

The above  results are used in the 
preprint \cite{FM} by P.~Fong and M.~Maccan,
which studies a remarkable class 
of smooth projective surfaces in positive 
characteristics:
relatively minimal surfaces equipped with
an isotrivial elliptic fibration.  It turns out that
every such surface $S$ is obtained as 
a quotient $(E \times X)/G$,
where $E$ is an elliptic curve,  
$G \subset E$ a finite subgroup scheme, 
and $X$ a projective $G$-normal curve.
Thus,  $S$ is equipped with an action of 
the algebraic group
$E$ for which the quotient is the elliptic
fibration $f : S \to X/G = Y$.  When 
$G$ is diagonalizable,  the multiple fibers
and relative canonical bundle of $f$
are described in terms of the quotient
$\pi : X \to Y$ in \cite[\S 3]{FM},
based on Theorems \ref{thm:local}
and \ref{thm:hurwitz}.  It would be very
interesting to obtain a global description
of $G$-normal curves in this setting,
by gluing their local models in the spirit of
\cite{Pardini}.

Theorem \ref{thm:local} is well-known if
$G$ is constant; then $X$ is a smooth curve,
the morphism $\psi : U/H \to V$ is \'etale, 
and $\pi$ is a tamely ramified $G$-cover. 
In this setting, Theorem \ref{thm:hurwitz}
and (\ref{eqn:hur}) follow from
the classical Hurwitz formula. 
For an arbitrary finite group scheme $G$, 
the morphism $\psi$ may well be purely 
inseparable. Still, its dualizing sheaf is 
equipped with a natural section,  
and this reduces the proof of Theorem 
\ref{thm:hurwitz} to a local computation.
In view of this theorem, the branch locus 
of $\pi$ (i.e., the locus where it fails to 
be a $G$-torsor) can be read off the 
relative canonical sheaf $\omega_{X/Y}$.

In the general setting of a $G$-normal variety
$X$ over an arbitrary field, where $G$ 
is diagonalizable, we obtain versions of
Theorems \ref{thm:local} and \ref{thm:hurwitz}
over an open $G$-stable subset of $X$ with
complement of codimension at least $2$;
see Theorems \ref{thm:ls} and \ref{thm:can}.
In this setting, the relative dualizing 
sheaf $\omega_{\pi}$ is torsion-free and
satisfies $(S_2)$, so that it suffices
to determine it in codimension $1$;
see \cite{Kollar} for further aspects of
duality for torsion-free $(S_2)$ sheaves.

The layout of this paper is as follows.
Sections \ref{sec:fgsa} and \ref{sec:diag}
collect preliminary notions and results 
on actions of finite group schemes over 
a field, especially
diagonalizable ones. For these, we discuss
torsors and uniform cyclic covers in the 
sense of \cite{AV} (Remarks \ref{rem:torsor} 
and \ref{rem:cyclic}), 
which form the main ingredients of our local 
structure results. Our presentation has some
overlap with \cite[Sec.~3]{Haution}; we have
provided details in order to be self-contained.
The main result of Section \ref{sec:env}
is a $G$-normality criterion for a finite
group scheme $G$
(Proposition \ref{prop:Gnormal}), which takes 
a simpler form when $G$ is diagonalizable.
Under this assumption, we obtain our local
structure theorem by steps: in Section 
\ref{sec:lcs}, we first describe the fibres of 
the quotient morphism $\pi$ at points of
codimension $1$ for a $G$-normal variety $X$,
and then the corresponding semi-local rings 
of $X$; these may be viewed as equivariant 
analogs of discrete valuation rings 
(see Propositions \ref{prop:fibre} and
\ref{prop:local}). 
From this, we deduce a more global 
description of $X$ in codimension $1$ 
(Theorem \ref{thm:ls}, the main result of
Section \ref{sec:lcsc}). In turn, this
yields in Section \ref{sec:dualizing}
a formula \`a la Hurwitz for the dualizing 
sheaf of the quotient morphism 
$\pi$. Section \ref{sec:curves} 
contains applications to curves; 
in particular, we determine the arithmetic 
genus of projective $G$-normal curves,
and its equivariant analog 
(Proposition \ref{prop:genus}).
In Section \ref{sec:lrgs},
we extend our description of the dualizing
sheaf to the setting of linearly reductive group
schemes (Theorem \ref{thm:linred} and
Proposition \ref{prop:lr}).

\medskip

\noindent
{\bf Acknowledgements.} 
I am grateful to Pascal Fong,  Qing Liu, 
Matilde Maccan,  David Rydh,  and Dajano Tossici 
for interesting discussions and helpful comments 
on earlier versions of this paper.

\section{Finite group schemes and their actions}
\label{sec:fgsa}

Throughout this paper, we work over a field $k$
of characteristic $p > 0$. 
Schemes are assumed to be separated and 
of finite type over $k$ unless explicitly
mentioned. A \emph{variety} 
is a geometrically integral scheme. 
We say that an open subset $U$ of a variety
$X$ is \emph{big} if its complement 
$X \setminus U$ has codimension at least $2$.

We denote by $G$ a finite group scheme and 
by $\vert G \vert$ its order, i.e., 
the dimension of the algebra 
$\cO(G)$ viewed as a $k$-vector space. 
Recall that $G$ lies in a unique exact sequence
\begin{equation}\label{eqn:exact}
1 \longrightarrow G^0 \longrightarrow G
\longrightarrow \pi_0(G) \longrightarrow 1,
\end{equation}
where $G^0$ is infinitesimal and $\pi_0(G)$
is finite and \'etale. If $k$ is perfect, then
this sequence has a unique splitting 
(see e.g.~\cite[II.5.2.4]{DG}.

We say that $G$ is \emph{linearly reductive}
if every $G$-module is semi-simple. By Nagata's
theorem, this is equivalent to $G^0$ being of
multiplicative type and $\vert \pi_0(G) \vert$
being prime to $p$ (see \cite[IV.3.3.6]{DG}).

A $G$-\emph{scheme} is a scheme $X$ 
equipped with a $G$-action
\[ \alpha : G \times X \longrightarrow X, 
\quad (g,x) \longmapsto g \cdot x. \] 
This action is said to be \emph{faithful} 
if every nontrivial subgroup scheme of $G$ 
acts nontrivially.

For any $G$-scheme $X$ and any closed
subscheme $X'$, the morphism
$G \times X' \to X$, $(g,x) \mapsto g \cdot x$
is finite. We denote its schematic image by
$G \cdot X'$; this is the smallest closed
$G$-stable subscheme of $X$ containing $X'$.  
Also, we denote by $C_G(X')$ the 
\emph{centralizer} of $X'$, i.e., 
the largest subgroup scheme of $G$ that acts 
trivially on $X'$. 
In particular, for any closed point $x \in X$,  
we obtain the \emph{orbit} $G \cdot x$
and the centralizer $C_G(x)$.
These will be discussed at the end of the next
paragraph.

Given a field extension $K/k$ and a scheme
$X$, we denote by $X_K$ the $K$-scheme
obtained from $X$ by the base change 
$\Spec(K) \to \Spec(k)$. Then
$G_K$ is a finite $K$-group scheme
and the formation of (\ref{eqn:exact})
commutes with such base change.
Moreover, every $G$-action $\alpha$ on $X$ 
yields a $G_K$-action $\alpha_K$ on $X_K$,
which may be identified with a $G$-action on
the $k$-scheme $X_K$ via the isomorphism
\[ G_K \times_{\Spec(K)} X_K
\stackrel{\sim}{\longrightarrow}
G \times X_K.
\]
The canonical projection $\pr : X_K \to X$ 
is equivariant relative 
to the homomorphism $G_K \to G$. 
If $x \in X$ is a closed point with residue 
field $K = \kappa(x)$, then viewing $x$ 
as a $K$-point of $X_K$, the orbit 
$G_K \cdot x \subset X_K$ has schematic image 
$G \cdot x$ under the above projection. Moreover, 
we have an isomorphism of $G_K$-schemes 
$G_K \cdot x \simeq G_K/C_{G_K}(x)$
and the equality $C_{G_K}(x) = \Stab_G(x)$, 
the fibre at $x$ of the stabilizer 
$\Stab_G \subset G \times X$. 

Next, we discuss quotients by $G$; for this, 
we only consider $G$-schemes $X$ that admit 
a covering by open affine $G$-stable subsets. 
This assumption is satisfied if $G$ is 
infinitesimal (then every open subset is 
$G$-stable), or if $X$ is quasi-projective
(e.g., a curve). Then $X$ admits 
a categorical quotient 
$\pi : X \to Y = X/G$,
where $Y$ is a scheme, $\pi$ 
is finite and surjective, and the morphism
\begin{equation}\label{eqn:graph} 
\gamma = (\alpha,\pr_2) : G \times X
\longrightarrow X \times_Y X, \quad
(g,x) \longmapsto (g\cdot x, x)
\end{equation}
is surjective as well. Therefore, 
$\cA = \pi_*(\cO_X)$ is a coherent 
algebra over $\cB = \cO_Y$, equipped with 
a $G$-linearization. 
Moreover, the natural map
$\cB \to \cA^G$ is an isomorphism
(see \cite[III.2.6.1]{DG} for these results).
If $X$ is a variety, then so is $Y$.

Also, recall that $X$ has a largest open
$G$-stable subset $X_{\fr}$ on which $G$
acts freely. Moreover, $\pi$ restricts to 
a $G$-torsor $X_{\fr} \to Y_{\fr}$, where 
$Y_{\fr}$ is open in $Y$ (see loc.~cit.). 
We say that a $G$-variety $X$ is 
\emph{generically free} if $X_{\fr}$ is
nonempty.

The quotient morphism $\pi$ is the composition
\begin{equation}\label{eqn:quotients}
X \stackrel{\varphi}{\longrightarrow}
Z = X/G^0 \stackrel{\psi}{\longrightarrow} Y,
\end{equation}
where $\varphi$ (resp.~$\psi$) denotes the
quotient by $G^0$ (resp.~$\pi_0(G)$).
Moreover, $\varphi$ is purely inseparable,
and hence a universal homeomorphism.
The formation of $\pi$, $\varphi$ and $\psi$
commutes with flat base change $Y' \to Y$;
in particular, with field extensions.

\begin{lemma}\label{lem:fact}
Assume that $X$ is a $G$-variety. Then
the $G$-action on $X$ is faithful 
(resp.~generically free, free) 
if and only if the $G^0$-action on $X$ 
is faithful (resp.~generically free, free) 
and the $\pi_0(G)$-action on $X/G^0$ 
is faithful (resp.~free).
\end{lemma}

\begin{proof}
By fpqc descent, we may assume $k$ 
algebraically closed. Then 
$G = G^0 \rtimes \pi_0(G)$ and hence
$G(k) \stackrel{\sim}{\longrightarrow}
\pi_0(G)(k)$; moreover, $\pi_0(G)$ 
is constant.

Denote by $H$ the kernel of the $G$-action
on $X$. Then also  
$H \simeq H^0 \rtimes \pi_0(H)$,
and $\pi_0(H)$ is constant. Moreover,
$H^0$ (resp.~$\pi_0(H)$) is the kernel 
of the action of $G^0$ (resp.~$\pi_0(G)$)
on $X$. Since $\varphi$ is bijective,
it follows that $\pi_0(H)$ is the kernel
of the $\pi_0(G)$-action on $Z$. This 
readily implies the assertion on faithful
actions.

The assertion on free actions is obtained 
similarly by considering the stabilizers 
of $k$-rational points of $X$.
This assertion implies that the $G$-action
on $X$ is generically free if and only if
so are the $G^0$-action on $X$ and the 
$\pi_0(G)$-action on $Z$. But for an action 
of a finite group (that is, a finite constant 
group scheme) on a variety, 
being generically free is equivalent to being 
faithful; this completes the proof.
\end{proof}

\begin{lemma}\label{lem:sub}
Let $X$ be a $G$-variety with quotient 
$\pi : X \to Y$, and let $X' \subset X$ be 
a closed $G$-stable subscheme. 
\begin{enumerate}
\item[{\rm (i)}] The quotient morphism
$\pi' : X' \to Y' = X'/G$ exists and lies in 
a commutative square
\[ 
\xymatrix{
X' \ar[r]^i \ar[d]_{\pi'} & X \ar[d]^{\pi} \\
Y' \ar[r]^j & Y, \\
} \]
where $i$ denotes the inclusion,  and
$j$ is finite and purely inseparable. 
If $G$ is linearly reductive,  then $j$ is
a closed immersion.
\item[{\rm (ii)}] For any open subset 
$U \subset X$ such that $U \cap X'$
is dense in $X'$,  there exists an open
$G$-stable subset $V \subset U$ such that
$V \cap X'$ is dense in $X'$.
\end{enumerate}
\end{lemma}

\begin{proof}
(i) Since $X' \subset X$ is covered by 
$G$-stable open affine subsets,  it admits
a categorical quotient $\pi' : X' \to Y'$. 
Moreover, the universal property of $\pi$
yields a unique morphism $j : Y' \to Y$
such that $j \circ \pi' = i \circ \pi$;
as a consequence,  $j \circ \pi'$ is finite.
Since $\pi'$ is dominant,  it follows that
$j$ is finite. Also,  for any algebraically 
closed field extension $K/k$, 
the map $j(K) : Y'(K) \to Y(K)$ is injective,
since $Y'(K)$ is the orbit space
$X'(K)/G(K)$ and likewise 
$Y(K) = X(K)/G(K)$. 
So $j$ is purely inseparable.

To show the final assertion,  we may assume
$X$ affine. Then the restriction map
$\cO(X) \to \cO(X')$ is surjective, 
and hence so is 
$\cO(X)^G \to \cO(X')^G$ 
by linear reductivity. Equivalently, 
$j^{\#} : \cO(Y) \to \cO(Y')$ is surjective
as desired.

(ii) If $G$ is constant,  then it permutes the
generic points $x_1, \ldots,x_n$ of $X'$.
By assumption,
$U \cap X' \supset \{ x_1, \ldots, x_n \}$
and hence $V = \bigcap_{g \in G} g \cdot U$
is an open $G$-stable subset of $X$ 
such that 
$V \cap X' \supset \{ x_1, \ldots, x_n \}$.
This yields our assertion in this case,
and hence in the case where $G$ is \'etale
by Galois descent.  The general case follows
from this by using the factorization 
(\ref{eqn:quotients}) of $\pi$.
\end{proof}

With the assumption of Lemma \ref{lem:sub},  
we denote by $x_1, \ldots, x_n$ the generic 
points of $X'$ and let 
\begin{equation}\label{eqn:semiloc}
\cO_{X,X'} = 
\cO_{X, x_1} \cap \cdots \cap \cO_{X,x_n}.  
\end{equation}
This is a semi-local ring, and its maximal 
ideals correspond bijectively to 
$x_1,\ldots,x_n$.

\begin{lemma}\label{lem:semiloc}
With the above notation, the ring $\cO_{X,X'}$ 
is equipped with a (functorial) $G$-action, 
and is a union of finite-dimensional $G$-submodules.
\end{lemma}

\begin{proof}
We have
$\cO_{X,X'} = \varinjlim \cO(U)$,
where $U$ runs over the open subsets of
$X$ containing $x_1,\ldots, x_n$.
By Lemma \ref{lem:sub},  it follows that
$\cO_{X,X'} = \varinjlim \cO(V)$,
where $V$ runs over the above open subsets 
which are $G$-stable.  This readily yields the
assertions.
\end{proof}

In what follows, we will consider $G$-varieties
with a prescribed quotient $Y$, that is, 
pairs $(X,\pi)$, where $X$ is a $G$-variety and 
$\pi : X \to Y$ is the quotient morphism. 
Assuming in addition that the $G$-action is 
faithful, we say that $X$ is a 
\emph{$G$-variety over $Y$}.
Morphisms of such pairs are $G$-equivariant
morphisms of varieties over $Y$; in particular, 
the automorphism group of $(X,\pi)$ is 
the group $\Aut^G_Y(X)$ of equivariant 
relative automorphisms. We now obtain 
a description of this group, under 
assumptions that will be fulfilled by 
diagonalizable groups.

\begin{lemma}\label{lem:aut}
Let $Y$ be a variety, and $X$ a $G$-variety
over $Y$. 

\begin{enumerate}
\item[{\rm (i)}] The quotient morphism
$\varphi : X \to X/G^0 = Z$ induces
an injective homomorphism 
$\varphi_* : \Aut_Y^{G^0}(X) \to \Aut_Y(Z)$.
\item[{\rm (ii)}] The natural homomorphism
$\pi_0(G)(k) \to \Aut_Y(Z)$ is bijective.
\item[{\rm (iii)}] If $G$ is commutative and
the exact sequence (\ref{eqn:exact}) splits, 
then
\[ G(k) \stackrel{\sim}{\longrightarrow}
\Aut_Y^{G^0}(X) = \Aut_Y^G(X). \]
\end{enumerate}

\end{lemma}

\begin{proof}
The existence of the homomorphism 
$\varphi_*$ follows from the universal 
property of the categorical quotient.
Clearly, the kernel of $\varphi_*$ is
isomorphic to $\Aut_Z(X)$, and hence is
a subgroup of $\Aut_{k(Z)}(k(X))$. Since 
the extension of function fields
$k(X)/k(Z)$ is purely inseparable, it follows 
that $\varphi_*$ is injective, proving (i).

By Lemma \ref{lem:fact}, the action of
$\pi_0(G)$ on $Z$ is generically free. 
Therefore, we may identify $\Aut_Y(Z)$ with
a subgroup of $\Aut_{Y_{\fr}}(Z_{\fr})$,
where $Z_{\fr} \to Y_{\fr}$ is a 
$\pi_0(G)$-torsor. Moreover,
we have natural isomorphisms
\[ 
\Aut_{Y_{\fr}}(Z_{\fr}) \simeq 
\Hom(Z_{\fr}, \pi_0(G)) \simeq
\Hom(\pi_0(Z_{\fr}), \pi_0(G))
\simeq \pi_0(G)(k),
\]
where the second one follows from 
\cite[I.4.6.5]{DG}, and the third one 
holds since $Z_{\fr}$ is a variety
and hence $\pi_0(Z_{\fr}) = \Spec(k)$.
This implies (ii).

Under the assumptions of (iii),  the group
$G(k)$ acts on $X$ by $G$-equivariant
automorphisms over $Y$, and 
$G(k) \stackrel{\sim}{\longrightarrow}
\pi_0(G)(k)$. This yields the desired
assertion by combining (i) and (ii).
\end{proof}

\section{Diagonalizable group schemes}
\label{sec:diag}

We assume that $G$ is diagonalizable, 
and denote by 
$\Lambda = \Lambda(G) = \Hom_{\gp}(G,\bG_m)$ 
its character group. Recall that $\Lambda$ 
is a finite abelian group, and we have 
a $\Lambda$-grading
\begin{equation}\label{eqn:algebra}
\cO(G) = 
\bigoplus_{\lambda \in \Lambda} k \lambda.
\end{equation}
We say that $G$ is \emph{cyclic} if
$\Lambda = \bZ/n\bZ$ for some positive
integer $n$; equivalently, $G$ is the
group scheme $\bfmu_n$ of $n$th roots
of unity.

For an arbitrary diagonalizable group $G$, 
the exact sequence
(\ref{eqn:exact}) has a unique splitting:
$G = G^0 \times \pi_0(G)$, where $G^0$ 
(resp.~$\pi_0(G)$) is a product of 
finitely many cyclic groups of order 
a power of $p$ (resp.~prime to $p$).  
Moreover, the subgroup schemes $H$ of $G$ 
correspond bijectively to the subgroups 
of $\Lambda$ via
$H \mapsto \Lambda(G/H) = \Lambda^H$.
As a consequence,  for any field extension
$K/k$, the map $H \mapsto H_K$ yields
a bijection between the subgroup schemes
of $G$ and those of $G_K$.

Given a $G$-scheme $X$ with quotient
$\pi : X \to Y$, the $G$-linearized sheaf 
$\cA$ decomposes into eigenspaces:
\begin{equation}\label{eqn:grading}
 \cA = \bigoplus_{\lambda \in \Lambda}
\cA_{\lambda}, 
\end{equation}
where $\cA_0 = \cB$ and each $\cA_{\lambda}$ 
is a coherent $\cB$-module. This yields again
a $\Lambda$-grading, i.e., the multiplication 
of $\cA$ induces morphisms of $\cB$-modules 
\begin{equation}\label{eqn:mult}
\mult_{\lambda,\mu} : 
\cA_{\lambda} \otimes_{\cB} \cA_{\mu}
\longrightarrow \cA_{\lambda + \mu}
\end{equation}
for all $\lambda,\mu \in \Lambda$, and 
\begin{equation}\label{eqn:sec}
\sigma_{\lambda,n} : \cA_{\lambda}^{\otimes n}
\longrightarrow \cA_{n\lambda}
\end{equation}
for all such $\lambda$ and all integers 
$n \geq 1$. The formation of 
$\cA_{\lambda}$, $\mult_{\lambda,\mu}$ 
and $\sigma_{\lambda,n}$ 
commutes with arbitrary base change $Y' \to Y$.

\begin{lemma}\label{lem:torsor}
With the above notation, $\pi$ is a $G$-torsor
if and only if $\mult_{\lambda,\mu}$ is an 
isomorphism for all $\lambda,\mu \in \Lambda$.
Under these assumptions, each
$\sigma_{\lambda,n}$ is an isomorphism as well.
\end{lemma}

\begin{proof}
The first assertion follows from 
\cite[Exp.~VIII, Prop.~4.1]{SGA3};
we present a proof for completeness.
If $\pi$ is the trivial torsor 
$\pr_Y : G \times Y \to Y$,
then $\cA$ is identified with
$\cB \otimes_k \cO(G)$ as a $\cB$-$G$-algebra,
where $G$ acts on $\cO(G)$ via its action on
itself by multiplication. Using the 
decomposition (\ref{eqn:algebra}), 
this identifies each multiplication map 
$\mult_{\lambda,\mu}$ with the map
$\cB \otimes_k k \lambda \otimes_k k \mu
\to \cB \otimes_k k(\lambda + \mu)$
induced by the multiplication in $\cO(G)$.
As a consequence, $\mult_{\lambda,\mu}$ is an
isomorphism. By fpqc descent, this still holds 
if $\pi$ is an arbitrary $G$-torsor.

Conversely, assume that each 
$\mult_{\lambda,\mu}$ is an isomorphism.
Then the multiplication map induces 
an isomorphism 
$\cA_{\lambda} \otimes_{\cB} \cA_{- \lambda}
\stackrel{\sim}{\longrightarrow} \cB$ 
for any $\lambda \in \Lambda$. In particular,
every $\cA_{\lambda}$ is invertible, 
and hence $\pi$ is finite and flat.
To show that it is a $G$-torsor, it suffices
to check that the morphism
(\ref{eqn:graph}) is an isomorphism.
As in the proof of Lemma \ref{lem:aut},
we may assume that $X = \Spec(A)$ and 
$Y = \Spec(B)$. Then $\gamma$ corresponds
to the algebra homomorphism
\[ 
\gamma^{\#} : A \otimes_B A 
\longrightarrow \cO(G) \otimes_k A,
\quad a_1 \otimes a_2 \longmapsto
(g \mapsto (g \cdot a_1) a_2),
\]
where $\cO(G) \otimes_k A$ is identified
with the space of morphisms $G \to A$.
Therefore, 
$\gamma^{\#} = 
\bigoplus_{\lambda,\mu \in \Lambda}
\gamma^{\#}_{\lambda,\mu}$,
where 
$\gamma^{\#}_{\lambda,\mu} : 
A_{\lambda} \otimes_B A_{\mu} \to
k \lambda \otimes_k A_{\lambda + \mu}$
satisfies
$\gamma^{\#}_{\lambda,\mu}
= \lambda \otimes \mult_{\lambda,\mu}$.
Thus, each $\gamma^{\#}_{\lambda,\mu}$
is an isomorphism, and hence so is
$\gamma$. 

This completes the proof of the first 
assertion ; the second assertion is 
obtained by a similar descent argument.
\end{proof}

\begin{remark}\label{rem:torsor}
The above lemma takes a much more concrete 
form when $G = \bfmu_n$. Then for any 
$G$-torsor $\pi : X \to Y$, we have
\[ \cA = \bigoplus_{m = 0}^{n-1} \cA_m, \]
where each $\cA_m$ is an invertible
$\cB$-module.
The multiplication map yields isomorphisms
$\sigma_{1,m} : \cA_1^{\otimes m} 
\stackrel{\sim}{\longrightarrow} \cA_m$
for $m = 1, \ldots, n - 1$, together
with an isomorphism 
$\sigma = \sigma_{1,n} : \cA_1^{\otimes n} 
\stackrel{\sim}{\longrightarrow} \cB$.
Moreover, one may easily check that the 
resulting homomorphism of $\cB$-algebras
$\Sym_{\cB}(\cA_1) \to \cA$
factors through an isomorphism
\[ \Sym_{\cB}(\cA_1)/\cI 
\stackrel{\sim}{\longrightarrow} \cA, \]
where $\cI$ denotes the ideal of 
$\Sym_{\cB}(\cA_1)$ generated by the 
$a^n - \sigma(a^{\otimes n})$
for $a \in \cA_1$. 

Conversely, given an invertible sheaf $\cL$ 
on $Y$ together with a trivializing section 
$\sigma \in \Gamma(Y,\cL^{\otimes n}) 
= \Hom_{\cO_Y}(\cL^{\otimes -n}, \cO_Y)$, 
the $\cB$-algebra
$\cA = \Sym_{\cB}(\cL^{\otimes -1})/\cI$
(where $\cI$ is defined as above)
corresponds to a $\bfmu_n$-torsor satisfying
$\cA_1 = \cL^{\otimes -1}$. 
This gives back a well-known classification of 
$\bfmu_n$-torsors, see 
e.g.~\cite[III.4.5.6]{DG}. It may be extended 
to torsors under diagonalizable group schemes
with only notational difficulties.
\end{remark}

\begin{remark}\label{rem:cyclic}
More generally, consider a scheme $Y$
equipped with an invertible sheaf $\cL$
and a section 
$\sigma \in \Gamma(Y,\cL^{\otimes n})$,
where $n$ is a positive integer (we no
longer assume that $\sigma$ trivializes
$\cL$). Denote again by $\cI$ the ideal of
$\Sym_{\cB}(\cL^{\otimes -1})$ generated 
by the $a^n - \sigma(a^{\otimes n})$,
where $a \in \cL^{\otimes -1}$, and let 
$\cA = \Sym_{\cB}(\cL^{\otimes -1})/\cI$.
Then $\cA$ is $\bZ/n\bZ$-graded and
$\cA_m = \cL^{\otimes -m}$ for 
$m = 0,\ldots, n-1$. So 
$X = \Spec_{\cB}(\cA)$ is a 
$\bfmu_n$-scheme with quotient $Y$, and 
$\sigma = \sigma_{1,n}$; moreover,
the quotient morphism $\pi : X \to Y$
is flat.

Denote by $E$ the zero scheme of 
$\sigma$; this is an effective Cartier 
divisor on $Y$. Moreover, the ideal
sheaf $\cJ \subset \cA$ generated by 
$\cL^{\otimes -1}$ is homogeneous and
satisfies 
$\cA/\cJ \simeq \cO_Y/\sigma \cO_Y
= \cO_Y/\cO_Y(-E)$. 
Thus, $\cJ$ corresponds to an effective 
$\bfmu_n$-stable Cartier
divisor $D$ on $X$, sent isomorphically
to $E$ by $\pi$. Moreover, $D$ is the 
fixed point subscheme $X^{\bfmu_n}$,
and $\pi$ restricts to a $\bfmu_n$-torsor
$X \setminus \Supp(D) \to 
Y \setminus \Supp(E)$. We have
$\pi^* \cO_Y(-E) = \cO_X(-nD)$
in view of the definition of $\cA$;
equivalently, $\pi^*(E) = nD$. 

The above construction is classical 
if $n$ is prime to $p$; then 
$\pi : X \to Y$ is \'etale outside of
$D$, and totally ramified along $D$. 
For an arbitrary $n$, this yields 
exactly the \emph{uniform cyclic covers}
considered in \cite{AV}.
\end{remark}

\begin{lemma}\label{lem:faithful}
Let $X$ be a $G$-variety. 

\begin{enumerate}
\item[{\rm (i)}] 
Each $\cA_{\lambda}$ is a torsion-free
$\cB$-module of rank $1$, and 
$\Lambda(X) = \{ \lambda \in \Lambda
~\vert~ \cA_{\lambda} \neq 0 \}$
is a subgroup of $\Lambda$. 
\item[{\rm (ii)}]
The $G$-action on $X$ is faithful if and 
only if $\Lambda(X) = \Lambda$; under these
assumptions, the action is generically free.
\end{enumerate}
\end{lemma}

\begin{proof}
(i) We may assume again that $X = \Spec(A)$
and $Y = \Spec(B)$. Then
$A = \bigoplus_{\lambda \in \Lambda}
A_{\lambda}$ 
is a $\Lambda$-graded domain, hence
$B$ is a domain and each $A_{\lambda}$ 
is a torsion-free $B$-module. Denoting
by $K$ (resp.~$L$) the fraction field
of $A$ (resp.~$B$), we have 
$K = A \otimes_B L$ (see 
e.g.~\cite[Lem.~2.3 (ii)]{Br24}) and hence
$K = \bigoplus_{\lambda \in \Lambda}
K_{\lambda}$, where $K_0 = L$ and
$K_{\lambda} = A_{\lambda} \otimes_B L$.
Thus, $\lambda \in \Lambda(X)$ if and
only if $K_{\lambda} \neq 0$; then 
$K_{\lambda}$ is a $1$-dimensional 
$L$-vector space. Moreover,
if $K_{\lambda}$ and $K_{\mu}$ are
nonzero, then also 
$K_{\lambda -\mu}$ is nonzero. This 
implies the assertion.

(ii) Since $X$ is affine, the $G$-action
is faithful if and only if so is its linear 
representation in $A$; equivalently, the 
intersection of the kernels of the 
characters $\lambda \in \Lambda(X)$ is 
trivial. In turn, this is equivalent to 
the equality $\Lambda(X) = \Lambda$.
We may then choose finitely many
homogeneous elements 
$f_1,\ldots, f_n \in A$
of respective weights 
$\lambda_1, \ldots, \lambda_n$
which generate the group $\Lambda$.
Then  
$U = X_{f_1} \cap \cdots \cap X_{f_n}$
is a dense open $G$-stable subset of 
$X$, equipped with a $G$-equivariant
morphism 
$(f_1,\ldots,f_n) : U \to \bG_m^n$
where $G$ acts on $\bG_m^n$ via
$g \cdot (t_1,  \dots,  t_n) =
(\lambda_1(g) t_1,  \ldots, 
\lambda_n(g) t_n)$.
Since the latter action is free, 
we see that $G$ acts freely on $U$.
\end{proof}

\section{Equivariantly normal varieties}
\label{sec:env}

Let $G$ be a finite group scheme, and
$X$ a $G$-variety. Following 
\cite[Def.~4.1]{Br24}, we say that $X$ is 
$G$-\emph{normal} if every finite birational
morphism of $G$-varieties $X' \to X$
is an isomorphism. 

By loc.~cit., \S 4, this notion of equivariant
normality is related to the usual notion of 
normality as follows: we may view $G$ 
as a closed subgroup scheme of a smooth
connected algebraic group $G^{\#}$ 
(for example, we may embed $G$ in $\GL_n$ 
via the regular representation, where
$n = \vert G \vert$). Given a $G$-variety
$X$, the product $G^{\#} \times X$
is a variety equipped with an action of
$G^{\#} \times G$ via
$(a,g) \cdot (b,x) = (a b g^{-1}, g \cdot x)$.
Since $G^{\#}$ and $X$ are covered by open 
affine $G$-stable subsets, the quotient
$X^{\#} = (G^{\#} \times X)/G$ exists;
this is a variety equipped with an action
of $G^{\#}$ and a $G^{\#}$-equivariant 
morphism $\varphi : X^{\#} \to G^{\#}/G$.
The fibre of $\varphi$ at the base point
of the homogeneous space $G^{\#}/G$
is $G$-equivariantly isomorphic to $X$.
We say that $X^{\#}$ is the 
\emph{induced variety} $G^{\#} \times^G X$.

\begin{lemma}\label{lem:normal}
With the above notation, $X$ is $G$-normal
if and only if $X^{\#}$ is normal.
\end{lemma}

This is proved by the same argument as
\cite[Lem.~4.7]{Br24}, which deals with
the quotient $(G^{\#} \times X)/G^0$.
As an application, we show:

\begin{lemma}\label{lem:hom}
Let $Y$ be a variety, and $u : X \to X'$ 
a morphism of $G$-varieties over $Y$. 
If $X'$ is $G$-normal and generically free, 
then $u$ is an open immersion.
\end{lemma}

\begin{proof}
Denote by $\eta$ the generic point of $Y$; 
then $u$ induces a $G_{\kappa(y)}$-equivariant
morphism $u_{\eta} : X_{\eta} \to X'_{\eta}$
between generic fibres. Moreover, $X'_{\eta}$
is a $G_{\kappa(y)}$-torsor by assumption, and
$X_{\eta}$ is $G_{\kappa(y)}$-homogeneous in view
of \cite[Cor.~2.5]{Br24}. Thus, $u_{\eta}$
is an isomorphism, and hence $u$ is birational.
Also, $u$ is quasi-finite, since 
$\pi = \pi' \circ u$ is finite. We now
consider the induced varieties $X^{\#}$, 
$X'^{\#}$ and the induced morphism
$u^{\#} : X^{\#} \to X'^{\#}$, which is still
birational and quasi-finite. Since $X'^{\#}$
is normal (Lemma \ref{lem:normal}), 
$u^{\#}$ is an open immersion by Zariski's 
Main Theorem.  Thus, $u$ is an open immersion.
\end{proof}

\begin{remark}\label{rem:Tvar}
A finite group scheme $G$ is diagonalizable 
if and only if it is isomorphic to a subgroup
scheme of a torus $T \simeq \bG_m^n$. 
Then the assignement
$X \mapsto X^{\#} = T \times^G X$ 
defines a bijection
between $G$-normal varieties and normal
$T$-varieties equipped with an equivariant
morphism $\varphi$ to the quotient torus 
$T/G$, whose fibre at the base point is 
geometrically integral. The dimension of $X$ 
is the \emph{complexity} of the $T$-variety 
$X^{\#}$ (the transcendence degree of the 
field of $T$-invariant rational functions).

Given a torus $T$, there is a classification 
of normal $T$-varieties in terms of objects 
mixing convex and algebraic geometry, called 
proper polyhedral divisors and divisorial fans,
which generalize the cones and fans associated
with toric varieties. Over an algebraically 
closed field $k$ of characteristic $0$, 
this classification is obtained in \cite{AH} 
for affine varieties; the general case is
treated in \cite{AHS}. Over an arbitrary field,
normal affine $T$-varieties of complexity
$1$ are also classified by proper polyhedral
divisors, as shown in \cite{Langlois}. 
Yet we do not know how to deduce from this 
a full classification of $G$-normal varieties,
which incorporates the additional datum of
the morphism $\varphi$.
\end{remark}

\begin{remark}\label{rem:linred}
Returning to an arbitrary finite group 
scheme $G$, note that a $G$-variety $X$ 
is $G$-normal if and only if $X_K$ is 
$G^0_K$-normal for some separable field 
extension $K/k$ (see 
\cite[Cor.~4.8, Prop.~4.9]{Br24}).
If $G$ is linearly reductive, we may 
take for $K/k$ a finite Galois extension
such that $G^0_K$ is diagonalizable
and $\pi_0(G)_K$ is constant. Then 
the quotient morphism $X_K \to Y_K$
factors uniquely as $X_K \to Z_K \to Y_K$, 
where $Z_K = X_K/G^0_K$ is a normal variety
and $Z_K \to Y_K$ is a finite Galois cover
with group $\pi_0(G)(K)$,  of order prime to $p$. 
This reduces a number of questions on 
equivariantly normal varieties under 
linearly reductive group schemes to the 
diagonalizable case; see the final Section
\ref{sec:lrgs} for an illustration.
\end{remark}

Next, we obtain a criterion for $G$-normality 
of a $G$-variety $X$ in terms of 
the $\cB$-algebra $\cA$, with the notation 
of Section \ref{sec:fgsa}. 
Clearly, we may assume that $G$ acts
faithfully on $X$.

Let $D$ be a prime divisor on $X$.
Then $G \cdot D$ is a closed $G$-stable
subscheme of $X$,  of pure codimension $1$.
We denote by $\cO_{X,G \cdot D}$ the 
corresponding semi-local ring; it has Krull
dimension $1$ and is equipped with a 
(functorial) $G$-action
(Lemma \ref{lem:semiloc}).  
We have $\cO_{X,G \cdot D}^G = \cO_{Y,E}$,
where $E = \pi(D)$ is a prime divisor 
on $Y$. Moreover, $\cO_{X,G \cdot D}$ 
has a largest $G$-stable ideal 
$I = I_{G \cdot D}$, since $G \cdot D$ 
is the smallest $G$-stable subscheme of 
$X$ containing $D$. (We may think of
$\cO_{X,G \cdot D}$ as a ``$G$-local'' ring
of dimension $1$).

We say that $D$ is \emph{free} 
if it meets the free locus $X_{\fr}$,
and \emph{nonfree} otherwise. If $G$ acts
generically freely on $X$, then there are 
only finitely many nonfree divisors: 
the irreducible components of codimension $1$ 
of the nonfree locus $X \setminus X_{\fr}$.

If $G$ is diagonalizable with character group
$\Lambda$, then the algebra $\cO_{X,G \cdot D}$ 
is $\Lambda$-graded by Lemma 
\ref{lem:semiloc} again. 
Thus, $I_{G \cdot D}$ is the largest 
$\Lambda$-graded ideal of $\cO_{X,G \cdot D}$. 
Moreover, the $G$-action on $X$ is 
generically free (Lemma \ref{lem:faithful}).

\begin{proposition}\label{prop:Gnormal}
Let $G$ be a finite group scheme, 
$Y$ a variety, and $X$ a $G$-variety over $Y$. 
Then $X$ is $G$-normal if and only if 
it satisfies the following conditions:

\begin{enumerate}
\item[{\rm (i)}] The variety $Y$ is normal.
\item[{\rm (ii)}] The coherent sheaf of 
$\cB$-algebras $\cA$ satisfies 
$(S_2)$.
\item[{\rm (iii)}] For any nonfree divisor
$D$, the ideal $I_{G \cdot D}$ is principal. 
\end{enumerate}

If $G$ is diagonalizable, then {\rm (ii)} 
is equivalent to 

\begin{enumerate}
\item[{\rm (ii)'}] 
The sheaf $\cA_{\lambda}$ is divisorial
for any $\lambda \in \Lambda$.
\end{enumerate}

\end{proposition}

Recall that a coherent sheaf $\cF$ satisfies 
Serre's condition ($S_2$) if it has depth 
at least $2$ at every point of codimension 
at least $2$ of its support.  Before proving
Proposition \ref{prop:Gnormal}, 
we record two preliminary results which
can be extracted from \cite{EGA}.

\begin{lemma}\label{lem:S2}
Let $\cF$ be a coherent torsion-free sheaf
on a variety $Z$. 

\begin{enumerate}
\item[{\rm (i)}] $\cF$ satisfies ($S_2$) 
if and only if the restriction map 
$\cF(V) \to \cF(U)$
is an isomorphism for any
open subsets $U \subset V \subset Z$ 
such that $V$ is big in $U$.
\item[{\rm (ii)}] Given a finite surjective 
morphism of varieties 
$\pi: Z \to W,$ the sheaf $\cF$ satisfies
($S_2$) if and only if so does 
$\varphi_*(\cF)$.
\end{enumerate}

\end{lemma}

\begin{proof}
(i) This is a consequence of
\cite[Prop.~5.9.8 and Prop.~5.10.14]{EGA}
(see also \cite[\S 1]{Kollar}).

(ii) This follows from
\cite[Cor.~5.7.11]{EGA} (see also
\cite[Lem.~18]{Kollar}).
\end{proof}

\noindent
\emph{Proof of Proposition \ref{prop:Gnormal}}.
By \cite[Thm.~4.12]{Br24}, $X$ is $G$-normal
if and only if $\cO_X$ satisfies ($S_2$) 
and every $G$-stable ideal is invertible in 
codimension $1$. Together with the above lemma,
it follows that (ii) holds for every $G$-normal 
variety $X$. Moreover, (i) holds by loc.~cit.,
Lem.~4.5, and (iii) follows from the fact that
every invertible ideal of a semi-local ring
is principal. 

Conversely, assume (i), (ii) and (iii). 
Then $\cO_X$ satisfies ($S_2$) by Lemma
\ref{lem:S2} again. So it suffices to check
that every $G$-stable ideal
is invertible in codimension $1$.
In other words, for any prime divisor $D$ 
on $X$, every $G$-stable ideal $I$ of 
$\cO_{X,G \cdot D}$ is principal.

If $D$ is free, we may replace $X$ with 
its dense open $G$-stable subset $X_{\fr}$, 
and hence assume that $G$ acts freely on $X$. 
Then by fpqc descent, the assignement 
$Z \mapsto Z \times_Y X$
yields a bijection between closed subschemes
of $Y$ and closed $G$-stable subschemes
of $X$. As a consequence, we have
$I = I^G \cO_X$. Moreover, 
$I^G$ is identified with an ideal of 
the discrete valuation ring $\cO_{Y,E}$
(where $E = \pi(D)$), and hence is invertible. 
This yields our assertion in this case.

If $D$ is nonfree, then we may choose a 
generator $f$ of the ideal $I_{G \cdot D}$. 
Then $I \subset f \cO_{X,G \cdot D}$, and hence 
$I = f J$ where $J$ is a $G$-stable ideal
of $\cO_{X,G \cdot D}$; in particular, 
$I \subsetneq J$. So we have an exact sequence
$0 \to J/I \to \cO_{X,G \cdot D}/I 
\to \cO_{X,G \cdot D}/J \to 0$
of $\cO_{X,G\cdot D}$-modules of finite length.
This implies our assertion by induction on 
the length of $\cO_{X,G\cdot D}/I$ 
(this length is minimal if and only if 
$I = I_{G \cdot D}$).

It remains to prove the equivalence 
(ii)$\Leftrightarrow$(ii)' when $G$ is 
diagonalizable. In view of the decomposition
(\ref{eqn:grading}), $\cA$ satisfies $(S_2)$ if
and only if each $\cA_{\lambda}$ satisfies
$(S_2)$. As $Y$ is normal and $\cA_{\lambda}$ 
has rank $1$, this amounts to $\cA_{\lambda}$
being divisorial (see 
e.g.~\cite[Cor.~1.4 and Prop.~1.6]{Hartshorne}).
\qed

\medskip

Combining Lemma \ref{lem:faithful} and  
Proposition \ref{prop:Gnormal},  we readily
obtain:

\begin{corollary}\label{cor:flat} 
Let $G$ be a finite diagonalizable group scheme
acting faithfully on a $G$-normal variety
with quotient morphism $\pi : X \to Y$.
Then $\pi$ is flat if and only if $\cA_{\lambda}$ 
is invertible for any $\lambda \in \Lambda$.
\end{corollary}

In particular,  $\pi$ is flat if $Y$ is regular.

\section{Local structure in codimension 1}
\label{sec:lcs}

Throughout this section, we consider 
a finite diagonalizable group $G$ and
a $G$-variety $X$ over a normal variety $Y$.
Let again $D$ be a prime divisor on $X$, 
with image $E$ in $Y$. Then $E$ is a prime 
divisor on the normal variety $Y$, and 
hence $B = \cO_{Y,E}$ 
is a discrete valuation ring. Moreover, 
the semi-local ring $A = \cO_{X,G \cdot D}$ 
is a finitely generated, torsion-free 
$B$-module, and hence is free of finite rank.
Also, $G$ acts on $A$, and $B = A^G$. 

Denote by $\fm$ the maximal ideal of $A$
corresponding to $D$; then $\fn = \fm \cap B$
is the maximal ideal of $B$. 
The residue field $A/\fm$ is the function 
field $k(D) = \kappa(x)$, where $x$ 
denotes the generic point of $D$. Likewise, 
$B/\fn = k(E) = \kappa(y)$ for the generic point
$y = \pi(x)$ of $D$; the extension 
$\kappa(x)/\kappa(y)$ is finite.
Let $\bar{A} = A/\fn A$; this is an artinian 
$\kappa(y)$-algebra, equipped with a action of
$G_{\kappa(y)}$.
We have $\bar{A} = \cO(X_y)$, where $X_y$ 
denotes the fibre of $\pi$ at $y$. 
Moreover, the orbit $G_{\kappa(y)} \cdot x$ 
is a subscheme of the finite $\kappa(y)$-scheme 
$X_y$ and both have the same underlying 
topological space. In particular, 
$G_{\kappa(y)} \cdot x$ only depends on the
point $y \in Y$. Also, this orbit may be 
identified with $G \cdot x$, as 
$\cO(G \cdot x)$ is the image of the
natural homomorphism
$\cO(X_y) \to \cO(G) \otimes_k \kappa(x)$,
and the right-hand side is identified with 
$\cO(G_{\kappa(y)}) \otimes_{\kappa(y)} \kappa(x)$.

We will obtain a description of the
$G$-algebra $\bar{A}$, or equivalently of
the $G$-scheme $X_y$, under the assumption
that $G$ is diagonalizable. Then $A$ is a
$\Lambda$-graded $B$-module:
\begin{equation}\label{eqn:gradedbis} 
A = \bigoplus_{\lambda \in \Lambda} 
A_{\lambda},
\end{equation}
where $A_0 = B$ and each $A_{\lambda}$ is 
a free $B$-module of rank $1$
(Lemma \ref{lem:faithful}). 
This yields a $\Lambda$-grading
\begin{equation}\label{eqn:gradedter} 
\bar{A} = \bigoplus_{\lambda \in \Lambda} 
\bar{A}_{\lambda}
\end{equation}
such that each $\bar{A}_{\lambda}$
is a one-dimensional $\kappa(y)$-vector space.
Let $I = I_{G \cdot D}$ be the largest 
$\Lambda$-graded ideal of $A$. Clearly, 
we have $\fn A \subset I$, and hence $I$ 
is the preimage of the largest 
$\Lambda$-graded ideal $\bar{I}$ of
$\bar{A}$ under the quotient map 
$A \to \bar{A}$. Moreover, there is an 
exact sequence of $G_{\kappa(y)}$-modules
\begin{equation}\label{eqn:orbit}
0 \longrightarrow \bar{I}
\longrightarrow \bar{A}
\longrightarrow \cO(G_{\kappa(y)} \cdot x)
\longrightarrow 0,
\end{equation}
and $\bar{I}$ is nilpotent.

\begin{proposition}\label{prop:fibre}
\begin{enumerate}
\item[{\rm (i)}]
There exists a unique subgroup scheme 
$H \subset G$ such that
$\bar{A} = \bar{A}^{H_{\kappa(y)}} \oplus \bar{I}$.
\item[{\rm (ii)}]
The morphism
$ X_y/H_{\kappa(y)} 
= \Spec(\bar{A}^{H_{\kappa(y)}}) 
\to \Spec(\kappa(y))$
is a torsor under $(G/H)_{\kappa(y)}$.
\item[{\rm (iii)}]
We have 
$H_{\kappa(y)} = C_{G_{\kappa(y)}}(x)$ and 
$\Spec(\bar{A}^{H_{\kappa(y)}}) 
\simeq G_{\kappa(y)} \cdot x$
as $(G/H)_{\kappa(y)}$-schemes.
\end{enumerate}
\end{proposition}

\begin{remark}\label{rem:fibre}
The subgroup scheme $H$ only depends
on $y \in Y$, or equivalently on the prime
divisor $E$; we will denote it by $H(E)$ 
or by $H(D)$ according to the setup.

If the restriction $\pi\vert_D : D \to E$
is birational,  then $\kappa(y) = \kappa(x)$ and
hence $H_{\kappa(x)}$ is just the stabilizer
$\Stab_G(x)$. In the general case, 
as the above proposition is rather technical,
we illustrate it with two examples before
giving the proof.

Let $G = \bfmu_{p^2}$  so that 
$\Lambda =\bZ/p^2\bZ$. 
Consider the polynomial ring
$k[T_0,T_1,T_2]$ equipped with
the $\Lambda$-grading such that
$T_0, T_1,T_2$ are homogeneous
of respective weights $0,1,p$. Let 
\[ X = \Spec \big( k[T_0,T_1,T_2]/
(T_1^p - T_0 T_2) \big).  \]
This is a normal affine surface
containing the origin as its unique
nonsmooth point.  Moreover,
$G$ acts on $X$ (since 
$T_1^p - T_0 T_2$ is homogeneous
of weight $p$) with quotient
\[ Y = \Spec(k[T_0,T_2^p])
\simeq \bA^2.  \]
The zero subscheme of $(T_0,T_1)$
in $X$ is a $G$-stable nonfree divisor $D$,  
and its image $E$ in $Y$ is the zero 
subscheme of $T_0$. With the above 
notation, we obtain
\[ \kappa(x) = k(T_2) \supset k(T_2^p) = \kappa(y),
\quad \bar{A} = k(T_2)[T_1]/(T_1^p),  \]
so that $H = \bfmu_p$. Note that 
$G_{\kappa(y)} \cdot x = x$ but $x$ is 
not fixed by $G_{\kappa(y)}$ (or equivalently,
by $G$).

Next, let $X' \subset X$ be the zero subscheme 
of $T_2^p - T_0 - 1 \in \cO(X)^ G = \cO(Y)$.  
Then $X'$ is a $G$-stable curve containing 
$x' = (0,0,1)$ as its unique nonsmooth 
point. Moreover, $X'$ has quotient 
$Y' = \Spec(k[T_0]) \simeq \bA^1$, and
$x'$ has image $y' = 0$.
We obtain with an obvious notation
\[ \kappa(x') = \kappa(y'),  \quad 
\bar{A}' = k[T_1,T_2]/(T_1^p,T_2^p - 1),  \]
so that $H' = \bfmu_p$ again; also,
$x'$ is neither free nor $G$-stable. 
One may check that $X'$ is $G$-normal.

The latter example will be generalized 
in Section \ref{sec:curves} to all 
$G$-normal curves, where $G$ is 
diagonalizable and infinitesimal. 
\end{remark}

\medskip \noindent
\emph{Proof of Proposition \ref{prop:fibre}}.
(i) For any $\lambda \in \Lambda$, we have 
either $\bar{A}_{\lambda} \subset \bar{I}$, 
or every nonzero element of $\bar{A}_{\lambda}$
is invertible. (Indeed, the ideal
$\bar{A}_{\lambda} \bar{A}$ is $\Lambda$-graded,
and hence is either contained in $\bar{I}$
or equal to $\bar{A}$). Also, note that every 
nonzero element of $\bar{A}_{\lambda}$ is 
invertible if and only if $\bar{A}_{\lambda}$
meets $\bar{A}^{\times}$. As a consequence, 
the characters satisfying this condition
form a subgroup of $\Lambda$.
Thus,  there exists a unique subgroup scheme 
$H$ of $G$ such that
$\bar{A}_{\lambda}$ meets $\bar{A}^{\times}$
if and only if 
$\bar{A}_{\lambda} \subset \bar{A}^{H_{\kappa(y)}}$.
In view of (\ref{eqn:gradedter}), this yields
$\bar{A} = \bar{A}^{H_{\kappa(y)}} \oplus \bar{I}$.

(ii) This follows from the definition of $H$ 
in view of Lemma \ref{lem:torsor}.

(iii) By (i) together with the exact sequence
(\ref{eqn:orbit}), we obtain 
a $G_{\kappa(y)}$-equivariant isomorphism 
$\bar{A}^{H_{\kappa(y)}} 
\stackrel{\sim}{\longrightarrow}
\cO(G_{\kappa(y)} \cdot x)$.
As a consequence,  $H_{\kappa(y)}$ acts trivially 
on $G_{\kappa(y)} \cdot x$.  In particular, 
$H_{\kappa(y)} \subset C_{G_{\kappa(y)}}(x)$.
On the other hand, 
$C_{G_{\kappa(y)}}(x)$ fixes the image
of $x$ in $\Spec(\bar{A}^{H_{\kappa(y)}})$. 
In view of (ii), this gives the opposite inclusion
$C_{G_{\kappa(y)}}(x) \subset H_{\kappa(y)}$.
\qed

We may reformulate the above proposition
in geometric terms by considering
the factorization of $\pi : X \to Y$ as
\[ 
X \stackrel{\varphi}{\longrightarrow}
Z = X/H  \stackrel{\psi}{\longrightarrow}
X/G = Y,
\]
where $\varphi$ (resp.~$\psi$) is the
quotient by $H$ (resp.~$G/H$). Then 
$\varphi$ yields a morphism between 
fibres $\varphi_y : X_y \to Z_y$. 
Moreover, $Z_y = X_y/H_{\kappa(y)}$ by Lemma
\ref{lem:sub} together with the linear 
reductivity of $H$.

\begin{corollary}\label{cor:factorization}

\begin{enumerate}
\item[{\rm (i)}] The restriction of
$\varphi$ to $D$ is birational to its image.
\item[{\rm (ii)}] The morphism
$\psi$ is a $G/H$-torsor in a neighborhood 
of $y$. 
\item[{\rm (iii)}] The fixed point subscheme 
$X_y^{H_{\kappa(y)}}$ yields a section of $\varphi_y$.
\item[{\rm (iv)}] The divisor $D$ is free 
if and only if $H$ is trivial.
\end{enumerate}
\end{corollary}

\begin{proof}
(i) The decomposition 
$\bar{A} = \bar{A}^H \oplus \bar{I}$
induces a decomposition
$\bar{\fm} = \bar{\fm}^H \oplus \bar{I}$,
so that
$\kappa(x) = A/\fm = \bar{A}/\bar{\fm}
= \bar{A}^H/\bar{\fm}^H 
= A^H/\fm^H,$
where the latter equality follows from
the inclusion $\fn A^H \subset \fm^H$.
This yields our assertion.

(ii) This follows from Proposition
\ref{prop:fibre} together with the openness of
the free locus of the $G/H$-variety $Z$. 

(iii) We identify the algebra homomorphism 
$\varphi_y^{\#} : \cO(Z_y) \to \cO(X_y)$
with the inclusion 
$\bar{A}^{H_{\kappa(y)}} \subset \bar{A}$.
The latter inclusion has a retraction
$\bar{A} \to \bar{A}/\bar{I} 
= \bar{A}^{H_{\kappa(y)}}$,
where $\bar{I}$ is the largest 
$H_{\kappa(y)}$-stable ideal of $\bar{A}$ 
such that $H_{\kappa(y)}$ acts 
trivially on $\bar{A}/\bar{I}$, i.e., 
the ideal of $X_y^{H_{\kappa(y)}}$.

(iv) If $H$ is trivial, then 
$X_y \to \Spec(\kappa(y))$
is a $G_{\kappa(y)}$-torsor, and hence 
$x \in X_{\fr}$. Conversely, if $H$ is 
nontrivial then $D$ is nonfree by (iii).
\end{proof}

For a $G$-normal variety $X$, the above 
description of $\bar{A} = \cO(X_y)$ 
lifts to a description of 
$A = \cO_{X, G \cdot x}$:

\begin{proposition}\label{prop:local}
Assume that $X$ is $G$-normal. 

\begin{enumerate}
\item[{\rm (i)}]
The largest $\Lambda$-graded ideal $I$ of 
$A$ is generated by a homogeneous element 
$f$ of weight $\nu \in \Lambda$ such that 
the restriction $\nu \vert_H$ generates 
$\Lambda(H)$. Moreover, $\nu \vert_H$ 
is uniquely determined by $I$.
\item[{\rm (ii)}]
We have $H \simeq \bfmu_n$ for some $n \geq 1$,
and $A = \bigoplus_{m = 0}^{n -1} A^H f^m
\simeq A^H[T]/(T^n - g)$,
where $g = f^n$ satisfies $g \in A^H$ and 
$g A = \fn A$.
\item[{\rm (iii)}] 
We have the equality of $H$-fixed subschemes
$X_y^H = G_{\kappa(y)} \cdot x$,
and the inclusion
$G \cdot D \subset X^H$ with equality
in a $G$-stable neighborhood of $x$.
\end{enumerate}
\end{proposition}

\begin{proof}
We show (i) and (ii) simultaneously. 
The largest $\Lambda$-graded ideal $I$ is 
principal (Proposition \ref{prop:Gnormal}). 
Let $f \in A$ such that $I = A f$ and write 
$f = \sum_{\lambda} f_{\lambda}$, where
$f_{\lambda} \in I_{\lambda}$ for all 
$\lambda$. We claim that $I = A f_{\nu}$ 
for some $\nu$.

Indeed, $f \notin \fm I$, where $\fm$ 
denotes the maximal ideal of $x$ in $A$. 
So there exists $\nu$ such that 
$f_{\nu} \notin \fm I$. We have
$A f_{\nu} \subset I$; if this inclusion
is strict, then the annihilator $J$
of the $\Lambda$-graded $A$-module 
$I/Af_{\nu}$ is a proper $\Lambda$-graded 
ideal, and hence is contained in $I$. 
But we have the equality of localizations
$I_{\fm} = (A f_{\nu})_{\fm}$ and hence
$J_{\fm} = A_{\fm}$, which contradicts the
fact that $I_{\fm} = \fm_{\fm} \neq A_{\fm}$.
This proves our claim.

Let $f = f_{\nu}$ as in that claim
and let $g = f^n$, where $n = \vert H \vert$.
Then $g \in A^H$, since the character group 
$\Lambda(H)$ is killed by $n$. Thus, the 
image $\bar{f}$ of $f$ in $\bar{A}$ 
satisfies $\bar{f}^n \in \bar{A}^{H_{\kappa(y)}}$.
But $\bar{f} \in \bar{I}$ and 
$\bar{A} = \bar{A}^{H_{\kappa(y)}} \oplus \bar{I}$, 
so that $\bar{f}^n = 0$; equivalently,
$g \in \fn A$. Moreover, Proposition
\ref{prop:fibre} yields the equalities
$\bar{A} = \bar{A}^{H_{\kappa(y)}} \oplus \bar{I} 
= \bar{A}^{H_{\kappa(y)}} \oplus \bar{f} \bar{A}$
and hence 
\[
\bar{A} = \bar{A}^{H_{\kappa(y)}} +  \bar{f} \bar{A} 
= \bar{A}^{H_{\kappa(y)}} 
+ \bar{f} \bar{A}^{H _{\kappa(y)}}
+ \bar{f}^2 \bar{A} = \cdots = 
\sum_{m=0}^{n -1} 
\bar{f}^m \bar{A}^{H_{\kappa(y)}}.
\]
Taking weights, we obtain 
$\Lambda = \Lambda(G/H) + \bZ \nu$,
where $\Lambda(G/H) = \Lambda^H$. Also,
recall that $\Lambda(H) = \Lambda/\Lambda^H$; 
thus, this group is generated by $\nu \vert_H$.
As a consequence, $\nu \vert_H$ has order
$n$, and hence $H \simeq \bfmu_n$;
moreover,  $\Lambda$
is the disjoint union of the translates
$\Lambda^H + m \nu$ where 
$m = 0,\ldots, n -1$.
Also,  $\bar{f}$ is uniquely determined
up to multiplication by a homogeneous
unit of $\bar{A}$, so that $\nu$
is unique up to translation by $\Lambda^H$;
equivalently, $\nu \vert_H$ is unique.

For any $\lambda \in \Lambda^H$ and 
$m = 0,\ldots, n -1$,
we have 
$\bar{A}_{\lambda + m \nu} = 
\bar{A}_{\lambda} \bar{f}^m$,
and hence 
$A_{\lambda + m \nu} = A_{\lambda} f^m$
by Nakayama's lemma. As a consequence,
$A = \sum_{m=0}^{n -1}
A^H f^m$. 
Moreover, this sum is direct in view of
the above decomposition of $\Lambda$
into cosets of $\Lambda^H$.

It remains to show that $g A = \fn A$, 
or equivalently, $g A^H = \fn A^H$.
Clearly, we have $g A^H \subset \fn A^H$. 
Thus, it suffices to check that
$\dim_{\kappa(y)} A^H/gA^H = 
\dim_{\kappa(y)} A^H/\fn A^H$.
But we have isomorphisms of $\kappa(y)$-algebras
\[ A^H/\fn A^H \simeq \bar{A}^{H_{\kappa(y)}}
\simeq \bar{A}/\bar{I} \simeq A/I = 
A/ f A \simeq A^H/g A^H, \]
where the latter isomorphism follows from
the decomposition of $A$ as an $A^H$-module. 
This yields the desired equality of 
dimensions.

(iii) Note that $X^H$ is $G$-stable, 
and contains $D$ by Proposition
\ref{prop:fibre}; therefore, 
$G \cdot D \subset X^H$. 
Denote by $\cJ \subset \cO_X$ the
ideal sheaf of $X^H$, and by $J \subset A$ 
the corresponding ideal. 
Then $f \in J$, as it is an $H$-eigenvector 
of nonzero weight. Thus, we have
$I(G \cdot D) \subset J$; 
this implies our assertions.
\end{proof}

\begin{remark}\label{rem:sec}
Given $\lambda \in \Lambda$, 
consider the multiplication map
\[ \sigma_{\lambda,\vert G \vert} : 
A_{\lambda}^{\otimes \vert G \vert} 
\longrightarrow A_{ \vert G \vert \lambda} 
= A_0 = B \]
as in (\ref{eqn:sec}). With the notation of
Proposition \ref{prop:local}, write 
$\lambda = \mu + m \nu$, where 
$\mu \in \Lambda^H$ and $0 \leq m \leq n-1$.
Then the image of 
$\sigma_{\lambda, \vert G \vert}$
is the ideal of $B$ generated by 
$f^{\vert G \vert m} = g^{[G:H] m}$.

Indeed, we have $A_{\lambda} = A_{\mu} f^m$
by the above proposition, and the 
multiplication map 
$\sigma_{\mu,\vert G \vert}: 
A_{\mu}^{\otimes \vert G \vert} \to B$
is an isomorphism in view of Remark
\ref{rem:torsor} and Corollary 
\ref{cor:factorization}.
\end{remark}

\section{Local structure in codimension 1
(continued)}
\label{sec:lcsc}

We still consider a finite diagonalizable
group scheme $G$ and a $G$-normal variety $X$
over $Y$ (a normal variety).
To describe the local structure of $X$ in 
codimension $1$, we may freely replace it by 
any big open subset $U$ which is $G$-stable
and satisfies some additional assumptions.  
Here are those we will need at this stage.

\begin{lemma}\label{lem:open}
There exists a largest open subset 
$U \subset X$ satisfying the following 
properties:

\begin{enumerate}

\item[{\rm (i)}] The quotient $U/G$ is regular.

\item[{\rm (ii)}] The nonfree locus 
$U \setminus U_{\fr}$ is a disjoint union of
prime divisors $D$.

\item[{\rm (iii)}] For any such $D$, the ideal 
$\cI_{G \cdot D}$ is invertible.

\end{enumerate}

Moerover, $U$ is $G$-stable and 
big in $X$.
\end{lemma}

\begin{proof}
These properties are open, and hold 
in codimension $1$; this readily yields
that $U$ exists and is big in $X$.
To complete the proof, it suffices to
show that each of these properties 
defines a $G$-stable open subset. 
This clearly holds for (i) and (ii).
For (iii), this follows from fpqc descent
as this property is stable by 
automorphisms arising from
$k$-points of $G$.
\end{proof}

We now assume that $X$ satisfies the above 
properties (i), (ii), (iii); we then say that 
$X$ is \emph{strongly $G$-normal}.

To every prime divisor $D$ on $X$, 
we associate a cyclic subgroup scheme 
$H(D)  \subset G$ 
(where $x$ denotes the generic point of $D$)
and a weight $\nu(D) \in \Lambda$ 
(the weight of a homogeneous generator of 
the largest $\Lambda$-graded ideal of
$\cO_{X,G \cdot D}$), as in Proposition 
\ref{prop:local}; recall that
$\nu(D) \vert_{H(D)}$ generates 
$\Lambda(H(D))$.

Given a cyclic subgroup scheme $H$ of $G$ and
a weight $\nu \in \Lambda$ such that 
$\nu \vert_H$ generates $\Lambda(H)$, let
\[ 
U(H,\nu) = X_{\fr} \cup \bigcup_D G \cdot D,
\]
where the union runs over the nonfree
divisors $D$ on $X$ such that $H(D) = H$ and 
$\nu(D) \vert_{H(D)} = \nu \vert_H$.
Then $U(H,\nu)$ is an open $G$-stable subset 
of $X$, and its quotient $V(H,\nu) = U(H,\nu)/G$ 
is identified with an open subset of $Y$.
Also, note that $\nu\vert_H$
identifies $H$ with the cyclic group scheme
$\bfmu_n$, where $n = n(H,\nu)$; moreover,
we have
$\Lambda = \Lambda^H + 
\sum_{m=0}^{n -1} m \nu$ and
$n \nu \in \Lambda^H$. 
The multiplication in $\cA$ yields 
a morphism of $\cB$-modules
\[
\sigma = \sigma_{\nu,n} : \cA_{\nu}^{\otimes n} 
\longrightarrow \cA_{n \nu}.
\]
The restriction $\sigma \vert_{V(H,\nu)}$ 
may be viewed as a section of the invertible 
sheaf 
$\cA_{\nu}^{\otimes - n} 
\otimes_{\cB} \cA_{n \nu}$
on $V(H,\nu)$. We denote its divisor of zeroes by 
$\Delta(H,\nu)$; this is an effective 
divisor on the regular variety $V(H,\nu)$,
that we call the \emph{branch divisor}.

\begin{theorem}\label{thm:ls}
Let $X$ be a strongly $G$-normal variety.

\begin{enumerate}

\item[{\rm (i)}] We have
$X = \bigcup_{H,\nu} U(H,\nu)$, and
$U(H,\nu) \cap U(H',\nu') = X_{\fr}$
unless $H' = H$ and 
$\nu' \vert_H = \nu \vert_H$; then 
$U(H',\nu') = U(H,\nu)$.

\item[{\rm (ii)}] The quotient 
$U = U(H,\nu) \to V=V(H,\nu)$ factors as
$U \stackrel{\varphi_U}{\longrightarrow}
U/H \stackrel{\psi_U}{\longrightarrow}
U/G$,
where $\psi_U$ is a $G/H$-torsor. 
Moreover, the natural homomorphism
of $\cA^H$-algebras
\[
\Phi : \big( \cA^H \otimes_{\cB} 
\Sym_{\cB}(\cA_{\nu}) \big)/\cI
\longrightarrow \cA
\]
restricts to an isomorphism on $V$, 
where $\cI$ denotes the ideal of 
$\cA^H \otimes_{\cB} \Sym_{\cB}(\cA_{\nu})$
generated by the 
$a^n - \sigma(a^{\otimes n})$ 
with $a \in \cA_{\nu}$. 

\item[{\rm (iii)}] The branch divisor 
$\Delta= \Delta(H,\nu)$ 
is reduced and its support is 
$V \setminus V_{\fr}$.

\item[{\rm (iv)}] The fixed point subscheme
$U^H$ is an effective Cartier divisor on $U$,
and we have $U^H = \bigcup G \cdot D$
(union over the $G$-orbits of nonfree divisors of $U$).
Moreover, $\pi_U^*(\Delta) = n U^H$
as Cartier divisors on $U$.

\end{enumerate}
 
\end{theorem}

\begin{proof}
(i) All of this follows readily from the 
definition of $U(H,\nu)$.

(ii) For any $\lambda, \mu \in \Lambda$,
the multiplication map 
$\mult_{\lambda,\mu}$ (\ref{eqn:mult}) restricts
to a nonzero homomorphism of invertible
$\cB_V$-modules, and hence yields an isomorphism
\[ 
\left(\cA_{\lambda} \otimes_{\cB} \cA_{\mu}\right)
\vert_V
\stackrel{\sim}{\longrightarrow}
\cA_{\lambda + \mu} 
\vert_V (-D_{\lambda, \mu}),
\]
where $D_{\lambda, \mu}$ is an effective
divisor on $V$. By Lemma \ref{lem:torsor},
$D_{\lambda, \mu}$ is a sum of nonfree
divisors. Combining this lemma with 
Proposition \ref{prop:local}, it follows that
$D_{\lambda, \mu} = 0$ for all
$\lambda, \mu \in \Lambda^H$. This yields
the assertion on $\psi_U$ by using again
Lemma \ref{lem:torsor}. 

We now show that $\Phi_V$ is an isomorphism.
Note that the natural map
\[ \bigoplus_{m=0}^{n - 1}
\cA^H \otimes_{\cB} \cA_{m \nu}
\longrightarrow \cA^H \otimes_{\cB} 
\Sym_{\cB}(\cA_{\nu})/\cI \]
is an isomorphism on $V$. So it suffices
to show that the multiplication map
\[ 
\mult_{\lambda,m \nu} : 
\cA_{\lambda} \otimes_{\cB} \cA_{m \nu}
\longrightarrow \cA_{\lambda + m \nu}
\]
is an isomorphism on $V$ for all 
$\lambda \in \Lambda^H$ and 
$m = 0,\ldots,n -1$. But this map is
an isomorphism at the generic point of
every prime divisor $D$ meeting $U$,
in view of Proposition \ref{prop:local}. 
This yields the desired assertion by arguing
as in the above paragraph.

(iii) This follows from Proposition 
\ref{prop:local}, since $\sigma$ is 
identified with $g$ at the generic point 
of every $D$ as above. 

(iv) The ideal of $U^H$ is generated by
the homogeneous components of $\cA^H$ of
nonzero weight relative to $H$, and hence
by $\cA_{\nu}$ in view of (ii). 
In particular, this ideal is invertible; 
this yields the first assertion.

By Proposition \ref{prop:local} again,
we have $U^H \subset \bigcup \, G \cdot D$
with equality at the generic points.
As $U^H$ and $\bigcup_D \, G \cdot D$ are
Cartier divisors, they must coincide.
The equality of Cartier divisors 
$\pi_U^*(\Delta) = n U^H$ follows 
similarly from Proposition \ref{prop:local}.
\end{proof}

\begin{remark}\label{rem:ls}
The above theorem has a partial converse:
let $U$ be a $G$-variety with quotient 
$\pi: U \to V$, where $V$ is regular.
If $U$ satisfies the conditions (ii) and (iii) 
for some $H$, $\nu$, $\sigma$ and $\Delta$, 
then $U$ is $G$-normal. 

Indeed, by Proposition \ref{prop:Gnormal}, 
it suffices to show that the largest 
$\Lambda$-graded ideal of
$A = \cO_{U,G \cdot D}$ is principal for any 
nonfree divisor $D$. But we have 
$A = A^H[f]$, where $f \in A$ is homogeneous
and satisfies $f^n \in A^H$. Thus,
$\bar{A} = \bar{A}^H[\bar{f}]$, where
$\bar{f}^n = 0$. Moreover, the natural
morphism $\Spec(A^H) \to \Spec(B)$ is
a $G/H$-torsor, hence the morphism
$\Spec(\bar{A}^H) \to \Spec(B/\fn) 
= \Spec(\kappa(y))$ 
is a $(G/H)_{\kappa(y)}$-torsor. 
As a consequence, $\bar{f}$ generates 
the largest $\Lambda$-graded ideal of 
$\bar{A} = A/\fn A$. Since 
$f A \supset f^n A = \fn A$, 
it follows that $f$ generates the largest
$\Lambda$-graded ideal of $A$.
\end{remark}

\begin{remark}\label{rem:secbis}
For any $\lambda \in \Lambda$, 
the multiplication map
$\sigma_{\lambda,\vert G \vert} :
\cA_{\lambda}^{\otimes \vert G \vert}
\to \cB$ restricts to a global section of 
$(\cA_{\lambda}\vert_{Y_{\reg}})
^{\otimes - \vert G \vert}$.
By Remark \ref{rem:sec}, the divisor of 
this section on each $U(H,\nu)$
equals $[G:H] m \Delta(H,\nu)$, where 
$m = m(H,\nu)$ is the unique integer 
such that $0 \leq m \leq \vert H\vert -1$ 
and $\lambda - m \nu \in \Lambda^H$.
As a consequence, the class of 
the divisorial sheaf $\cA_{\lambda}$
in $\Cl(Y)_{\bQ}$ (the divisor class
group of $Y$ with rational coefficients)
satisfies
\begin{equation}\label{eqn:ratclass}
[\cA_{\lambda}] = - \sum_{H,\nu} 
\frac{m(H,\nu)}{\vert H \vert} \Delta(H,\nu)
\end{equation}
with the above notation.
\end{remark}

With the notation of Theorem \ref{thm:ls}
(ii), the morphism 
$\varphi_U: U \to U/H$ is a uniform
cyclic cover with branch divisor $\Delta$,
as defined in \cite{AV}.
We say that a $G$-normal variety $X$ is
\emph{uniform}, if it is the closure of some
$U = U(H,\nu)$; equivalently, the equality
$X = U(H,\nu)$ holds in codimension $1$.
(We do not assume that $X$ is strongly
$G$-normal). Then the closure of the branch 
divisor $\Delta$ in $Y$ is 
a reduced effective divisor that we will 
still call the branch divisor, and 
denote by $\Delta$ for simplicity.

\begin{proposition}\label{prop:uniform}
Let $X = \overline{U(H,\nu)}$ be 
a uniform $G$-normal variety over $Y$. 
Assume that $\cO(Y)^{\times} = k^{\times}$
and the divisor class group $\Cl(Y)$ 
has no $n$-torsion, where $n = \vert G \vert$. 
Then $G = H \simeq \bfmu_n$, the
$G$-variety $X$ over $Y$ is uniquely 
determined by $\Delta$, 
and the class $[\Delta]$ is divisible 
by $n$ in $\Cl(Y)$. Moreover, every reduced 
effective divisor on $Y$ with class divisible 
by $n$ is obtained from a uniform 
$\bfmu_n$-normal variety.
\end{proposition}

\begin{proof}
If $H \neq G$ then $G/H$ has a quotient
isomorphic to $\bfmu_d$ for some
$d$ dividing $n$, and hence the $G/H$-torsor 
$\psi : U/H \to V$ factors through a 
$\bfmu_d$-torsor $\phi: W \to V$, where $W$ 
is a variety. 
Since $\Pic(V) = \Cl(V) = \Cl(Y)$ has no
$d$-torsion, we have 
$\phi_*(\cO_W) = \cO_V[T]/(T^d -f)$
for some $f \in \cO(V)^{\times}$
(Remark \ref{rem:torsor}). But 
$\cO(V)^{\times} = \cO(Y)^{\times} 
= k^{\times}$ and hence $W$ is not
geometrically integral, a contradiction.

Thus, $G = H$ and we may assume that 
$G = \bfmu_n$ and $\nu$ is
the defining character. Then 
$\cA = \bigoplus_{m=0}^{n -1} \cA_m$,
where each $\cA_m$ is a divisorial
$\cB$-module (Proposition \ref{prop:Gnormal}).
By Theorem \ref{thm:ls}, the multiplication
map $\mult_{\ell,m} : 
\cA_\ell \otimes_{\cB} \cA_m 
\to \cA_{\ell + m}$
is an isomorphism on $V$ whenever
$\ell + m \leq n -1$. So we may assume
that $\cA_m = \cO_Y(-m D)$ for 
$m = 0, \ldots, n -1$, where $D$
is a Weil divisor on $Y$. Moreover,
the multiplication map 
$\sigma_{1,n} : \cA_1^{\otimes n} \to \cB$ 
yields an isomorphism 
$\cO_Y(- n D) 
\stackrel{\sim}{\longrightarrow}
\cO_Y(- \Delta)$
in view of Theorem \ref{thm:ls} again.
Equivalently, we have a section
$\sigma \in \Gamma(Y,n D)$
with divisor $\Delta$, which uniquely
determines the multiplication of the 
$\cB$-algebra $\cA$. Thus, the $G$-variety 
$X$ over $Y$ is uniquely determined by 
the pair $(D,\sigma)$ up to isomorphism
given by linear equivalence of divisors;
moreover, we have $[\Delta] = n [D]$
in $\Cl(Y)$. In turn, $X$ is uniquely 
determined by $\Delta$ in view of our 
assumptions on $Y$. This proves the 
second assertion; the final one follows 
from Theorem \ref{thm:ls} once more.
\end{proof}

\begin{remark}\label{rem:uniform}
Keep the asssumptions of the above proposition,
and assume in addition that $Y$ is regular.
We may further assume that $G = \bfmu_n$
and $\nu$ is the defining character.
Then the sheaf $\cL = \cA_1$ is invertible;
denote by 
\[ \varphi : L =\bV(\cL) =
 \Spec \,  \Sym_{\cB}(\cL) \longrightarrow Y \]
the line bundle with sheaf of local sections 
$\cL^{\otimes -1}$.  By the above proof,  
we may view $X$ as the zero scheme in $L$ 
of the section 
\[ \tau^n - \sigma  \in 
\Gamma(L, \varphi^*(L^{\otimes n}))
= \bigoplus_{m = 0}^{\infty} \, 
\Gamma(Y,\cL^{\otimes m - n}), \]
where $\tau \in \Gamma(L,\varphi^*(L))
= \bigoplus_{m = 0}^{\infty} \, 
\Gamma(Y,\cL^{\otimes m - 1})$
denotes the canonical section, corresponding
to $1 \in \Gamma(Y, \cO_Y)$, and 
$\sigma \in \Gamma(Y,\cL^{ \otimes -n})$.
This gives back a classical construction
of $\bfmu_n$-covers, see 
e.g.~\cite[Prop.~4.1.6]{Lazarsfeld}.

Also, note that for $p = 2$, every
$\bfmu_2$-normal variety is uniform.
But there exist nonuniform $\bfmu_p$-normal
varieties for any $p \geq 3$: for example,
$\bP^1$ where $\bfmu_p$ acts by multiplication.
\end{remark}

\begin{example}\label{ex:uniform}
Let $X$ be a uniform $G$-normal variety
over the projective space $\bP^N$.
By Proposition \ref{prop:uniform}, we have
$G \simeq \bfmu_n$ and $X$ is classified
by its branch divisor, a reduced hypersurface 
$\Delta \subset \bP^N$ of degree $d$
divisible by $n$. Moreover, $X$ is realized
as a hypersurface in the line bundle
$O_{\bP^N}(d) = \bV(\cO_{\bP^N}(-d))$ 
via the construction of Remark 
\ref{rem:uniform}.

If $X$ is regular,  then so is $X^G$
in view of \cite[Lem.~3.5.2]{Haution}.
Thus,  taking for $\Delta$ a singular
hypersurface yields many examples
of singular $G$-normal varieties.
\end{example}

\section{The relative dualizing sheaf}
\label{sec:dualizing}

Let $G$ be a finite group scheme acting
on a scheme $X$ with quotient morphism
$\pi: X \to Y$. Since $\pi$ is finite,
it has a $(0)$-dualizing sheaf 
$\omega_\pi = \pi^! \cO_X$; this is
the $\cO_X$-module that corresponds to
the $\cA$-module $\cHom_{\cB}(\cA,\cB)$.
It is equipped with a trace map 
$\tr_{\pi} : \pi_*(\omega_\pi) \to \cB$ 
via evaluation at $1$. 
If $\pi$ is flat and a local complete 
intersection (l.c.i.), then $\omega_\pi$ 
is isomorphic to the canonical sheaf
$\omega_{X/Y}$; in particular, 
$\omega_{\pi}$ is invertible (see 
e.g.~\cite[\S 6.4]{Liu}, which will be our 
reference for duality theory). In any case, 
$\omega_{\pi}$ is equipped with a 
$G$-linearization for which the trace map 
is invariant.

If $G$ is linearly reductive, then the
unique $G$-invariant projection 
$\cA \to \cB$ (the ``Reynolds operator'')
yields a global section 
$s_{\pi} \in \Gamma(X,\omega_{\pi})
= \Hom_{\cB}(\cA,\cB)$.
Note that $s_{\pi}$ is $G$-invariant
and satisfies $\tr_{\pi}(s_{\pi}) = 1$.
Also, the formation of $\omega_{\pi}$
and $s_{\pi}$ commutes with flat base
change $Y' \to Y$. If the morphism $\pi$ 
is flat l.~c.~i, then we denote $s_{\pi}$ by
$s_{X/Y}$.

Given a subgroup scheme $H \subset G$
and the corresponding factorization
of $\pi$ as 
$X \stackrel{\varphi}{\longrightarrow}
Z = X/H \stackrel{\psi}{\longrightarrow}
Y = X/G$,
we may identify 
$\omega_{\pi}$ with 
$\omega_{\varphi} \otimes_{\cO_X}
\varphi^*(\omega_{\psi})$
(see \cite[Lem.~6.4.26]{Liu}). 
If $G$ is linearly reductive, then we have
\begin{equation}\label{eqn:adj} 
s_{\pi} = 
s_{\varphi} \otimes \varphi^*(s_{\psi}),
\end{equation}
where 
$s_{\psi} \in \Gamma(Z,\omega_{\psi})
= \Hom_{\cB}(\cA^H,\cB)$ is the 
projection $\cA^H \to \cB$. If in addition
$H$ is a normal subgroup of $G$, then 
$\psi$ is the quotient morphism by $G/H$, and 
$s_{\psi}$ is the corresponding canonical section.

We will need two further observations,  certainly 
well-known but for which we could not locate 
any reference.

\begin{lemma}\label{lem:cantors}
Let $G$ be a finite group scheme, and
$\pi : X \to Y$ a $G$-torsor. 

\begin{enumerate}

\item[{\rm (i)}]
The (finite flat) morphism $\pi$ is l.c.i.  and satisfies
$\omega_{X/Y} \simeq \cO_X$. 

\item[{\rm (ii)}]
If $G$ is linearly reductive,  then $s_{X/Y}$ 
is a trivializing section of $\omega_{X/Y}$.

\end{enumerate}

\end{lemma}

\begin{proof}
(i) We use the construction at the beginning of
Section \ref{sec:env}.
Choose an embedding of $G$ in some $\GL_n$.
Then $\pi$ is the composition
\[
X \stackrel{\iota}{\longrightarrow}
Z = \GL_n \times^G  X
\stackrel{\psi}{\longrightarrow} Y, 
\]
where the morphism $\iota$ is a closed immersion 
with image $G \times^G X$,  and $\psi$
is the projection 
$\GL_n \times^G  X \to X/G = Y$.
Thus, $\psi$ is a $\GL_n$-torsor.
In particular, it is smooth and its 
relative tangent sheaf is isomorphic to 
$\Lie(\GL_n) \otimes_k \cO_{Z/Y}$.
As a consequence,
$\det(\Omega^1_{Z/Y}) \simeq \cO_{Z/Y}$.
Also, recall that $X$ is the fibre at the base
point of the other projection 
$\varphi : \GL_n \times^G  X \to \GL_n/G$.
Thus, $\iota$ is a regular immersion
and its conormal sheaf $\cC_{X/Z}$
is trivial (see \cite[Prop.~6.3.11]{Liu}).
This gives an isomorphism
$\omega_{X/Y} \simeq \cO_X$ in view of 
loc.~cit., Def.~6.4.7.

(ii) We use the commutative diagram of flat morphisms
\[ \xymatrix{
G \ar[r]  
& \Spec(k)  \\
G \times X 
\ar[r]^-{\pr_X} \ar[d]_{\alpha} \ar[u]^{\pr_G}
& X \ar[d]^{\pi} \ar[u] \\
X \ar[r]^-{\pi} & Y,
} \]
where both squares are cartesian.
Since the formation of $s_{X/Y}$ commutes
with flat base change, this yields
$\alpha^*(s_{X/Y}) = s_{G \times X/X}
= \pr_G^*(s_G)$. By fpqc descent, we are
reduced to checking that 
$s_G \in \Gamma(G,\omega_G) = \Hom_k(\cO(G),k)$ 
is a free generator of the $\cO(G)$-module
$\omega_G$. 

In the case where $G$ is diagonalizable,  
the $k$-vector space $\Hom_k(\cO(G),k)$
has basis
($e_{\lambda})_{\lambda \in \Lambda}$
dual to the basis $\Lambda$ of $\cO(G)$,
and the $\cO(G)$-module structure of
$\Gamma(G,\omega_G)$ is given by
$\mu \cdot e_{\lambda} = e_{\lambda - \mu}$.
Therefore, this module is freely generated
by $e_0 = s_G$.

In the general case of a linearly reductive
group $G$, we may assume $k$ 
algebraically closed by fpqc descent again.
Then $G^0$ is diagonalizable and
$G \simeq G^0 \rtimes \pi_0(G)$,
where $\pi_0(G)$ is constant of order $N$
prime to $p$.
Thus, $s_G : \cO(G) \to k$ satisfies
\[ s_G(\varphi) = \frac{1}{N} 
\sum_{g \in G(k)} g \cdot \varphi_0 \]
for any $\varphi \in \cO(G)$, 
where $\varphi_0$ denotes the 
homogeneous component of weight $0$
relative to $G^0$; note that
$g \varphi_0 = (g \cdot \varphi)_0$
as $g$ induces an automorphism of
$\Lambda(G^0)$.  Moreover,  we have
an isomorphism of algebras
\[ \cO(G) \simeq \prod_{g \in G(k)} \cO(gG^0) \]
and $s_G(\varphi) = \frac{1}{N} \varphi_0$
for any $g \in G(k)$ and $\varphi \in \cO(g G^0)$.
As $g G^0$ is a trivial $G^0$-torsor,  
$s_G$ restricts to a nonzero scalar multiple of 
$s_{gG^0}$. This implies our assertion. 
\end{proof}

\begin{lemma}\label{lem:canS2}
Let $G$ be a finite group scheme, 
and  $X$ a $G$-normal variety with quotient 
$\pi : X \to Y$. Then the sheaf $\omega_{\pi}$ 
is torsion-free and satisfies ($S_2$).
\end{lemma}

\begin{proof}
By local duality, we have an isomorphism
$\Hom_{\cO_X}(\cF,\omega_{\pi}) \simeq
\Hom_{\cB}(\pi_*(\cF),\cB)$
for any coherent sheaf $\cF$ on $X$.
If $\cF$ is torsion then so is $\pi_*(\cF)$;
this readily implies that $\omega_{\pi}$
is torsion-free. To show that $\omega_{\pi}$
is ($S_2$), it suffices to check that
so is $\pi_*(\omega_{\pi})$ (Lemma 
\ref{lem:S2}). As
$\pi_*(\omega_{\pi}) \simeq 
\cHom_{\cB}(\cA,\cB)$, the desired assertion
follows from the normality of $Y$ and
\cite[Cor.~1.2 and Prop.~1.3]{Hartshorne}.
\end{proof}

We also record a preliminary result,  to be
used in Sections \ref{sec:curves} and 
\ref{sec:lrgs}.

\begin{lemma}\label{lem:flatlci}
Consider again a finite group scheme $G$,  and 
a $G$-normal variety $X$ with quotient $\pi : X \to Y$. 
Then there exists a $G$-stable big open subset
$U \subset X$ such that the quotient $V = U/G$
is regular and
the morphism $\pi \vert_U : U \to V$ is flat l.c.i.
\end{lemma}

\begin{proof}
We may replace $X$ with $\pi^{-1}(Y_{\reg})$,
and hence assume that $Y$ is regular.
Then $\pi$ is flat in codimension $1$,  so that 
we may further assume that it is flat everywhere.
We now argue as in the proof of Lemma 
\ref{lem:cantors} (i).  With the notation of that proof,
the variety $Z$ is normal
by Lemma \ref{lem:normal},  and $Z_{\reg}$
is stable by $\GL_n$ as the latter is smooth.  
Thus,  we have $Z_{\reg} = \GL_n \times^G U$
for a unique $G$-stable open subset $U \subset X$.
Moreover,  $U$ is big in $X$,
since $Z_{\reg}$ is big in $Z$. 
To complete the proof,  it suffices to show that
$\pi\vert_U$ is l.c.i.  But
$\pi \vert_U = \psi \vert_{Z_{\reg}}  \circ \iota$,  
where $\iota : U \to Z_{\reg}$ is a regular immersion
and $\psi \vert_{Z_{\reg}} : Z_{\reg} \to Y$ 
is a morphism between regular varieties,  
and hence is l.c.i.  by \cite[Ex.~6.3.18]{Liu}.  
So the desired assertion follows from loc. cit., 
Prop.~6.3.20.
\end{proof}

Next, we assume that $G$ is diagonalizable
and $X$ is a $G$-normal variety on which
$G$ acts faithfully. Denote by 
$\iota : U \subset X$
the inclusion of the largest strongly 
$G$-normal subset (Lemma \ref{lem:open}).
Then we have 
$\omega_{\pi} = \iota_*(\omega_{\pi \vert_U})$
in view of Lemmas \ref{lem:S2} and 
\ref{lem:canS2}. Therefore, to determine
$\omega_{\pi}$, we may further assume that
\emph{$X$ is strongly $G$-normal}.

\begin{theorem}\label{thm:can}
With the above assumptions, $\pi$ is flat and
l.c.i.  Moreover,
\begin{equation}\label{eqn:can} 
\div(s_{X/Y}) = \sum \big( \vert H(D) \vert -1 \big) \, 
G \cdot D, 
\end{equation}
where the sum runs over the $G$-orbits
of nonfree divisors $D$.
\end{theorem}

\begin{proof}
The first assertion follows from Corollary 
\ref{cor:flat} and Theorem \ref{thm:ls} (ii).  
With the notation of this theorem,
$X$ is covered by the $G$-stable open
subsets $U(H,\nu)$, and hence we may 
assume that $X = U(H,\nu)$. Then
$\pi = \psi \circ \varphi$,
where $\psi : Z = X/H \to Y = X/G$ is 
a $G/H$-torsor.  Thus,  $s_{Z/Y}$ is 
a trivializing section of 
$\omega_{Z/Y}$ (Lemma \ref{lem:cantors}). 
Also,
$\omega_{X/Y} \simeq \omega_{X/Z}$
by \cite[Thm.~6.4.9]{Liu}; this identifies
$s_{X/Y}$ with $s_{X/Z}$ in view of
(\ref{eqn:adj}). So it suffices to 
check that the divisor of $s_{X/Z}$ equals 
$(\vert H \vert -1) \, \sum_D G \cdot D$
(sum over the $G$-orbits of nonfree divisors).
 
Let $D$ be such a divisor, with generic 
point $x$. With the notation of Proposition
\ref{prop:local}, the $A^H$-module $A$
is free with basis $1,f,\ldots,f^{n-1}$,
where $f$ is a local equation of
$G \cdot x$ at $x$, and $n = \vert H \vert$.
Denote by $e_0,\ldots,e_{n-1}$ the dual
basis of the $A^H$-module 
$\varphi^! A = \Hom_{A^H}(A,A^H)$; then
$e_0$ is the projection $A \to A^H$.
Moreover, the $A$-module structure of 
$\varphi^! A$ is given by 
$f \cdot e_i = e_{i-1}$ for 
$i = 1,\ldots, n-1$, and 
$f \cdot e_0 = e_{n-1}$. So 
$\varphi^! A$ is freely generated by 
$e_{n-1}$, and $e_0 = f^{n-1} e_{n-1}$.
In view of Proposition \ref{prop:local}
again,  it follows that 
$\div(s_{X/Y}) = (n-1) G \cdot D$
in a neighborhood of $x$.
\end{proof}

\begin{remark}\label{rem:can}
Still assuming $G$ diagonalizable and
$X$ strongly $G$-normal,  we may rewrite 
(\ref{eqn:can}) as the equality of Cartier divisors
\begin{equation}\label{eqn:div}
\div(s_{X/Y}) = \pi^*(\Delta_Y) - G \cdot \Delta_X,
\end{equation}
where $\Delta_Y = Y \setminus Y_{\fr}$
and 
$\Delta_X = X \setminus X_{\fr}$
are viewed are reduced effective divisors,
so that $\Delta_Y$ is the branch divisor.
Indeed,  
\[ G \cdot \Delta_X = \sum G \cdot D \]
(sum over the $G$-orbits of nonfree divisors);
this is a $G$-stable effective Cartier divisor 
on $X$.  Moreover,   
\[
\pi^*(\Delta_Y) = \sum \vert H(D) \vert \, G \cdot D
\]
in view of Theorem \ref{thm:ls} (iv).  

In turn,  (\ref{eqn:div}) gives an isomorphism of 
$G$-linearized  sheaves
\begin{equation}\label{eqn:log}
\omega_X(G \cdot \Delta_X)
\simeq \pi^*\big(\omega_Y(\Delta_Y)\big),
\end{equation}
where the dualizing sheaf $\omega_X$ is 
equipped with a $G$-linearization
via the canonical isomorphism
\begin{equation}\label{eqn:dualizing}
\omega_X \simeq \omega_{X/Y} \otimes 
\pi^*(\omega_Y).
\end{equation}
The isomorphism (\ref{eqn:log}) will be
used in the next section to determine
the ``equivariant arithmetic genus" of 
projective $G$-normal curves.  The equality
(\ref{eqn:div}) will be generalized to linearly 
reductive group schemes in Section 
\ref{sec:lrgs}.
\end{remark}

\section{Curves}
\label{sec:curves}

Throughout this section,  $G$ denotes 
a finite diagonalizable group scheme,
$Y$ a curve,  and $X$ a $G$-normal curve 
over $Y$.  Then $Y$ is regular,  the orbit 
$G \cdot x$ is an effective Cartier divisor 
on $X$ for any closed point $x \in X$, 
and the quotient
$\pi: X \to Y$ is flat l.~c.~i.~(Lemma 
\ref{lem:flatlci}).  Also,  the $G$-action on $X$ 
is generically free; thus,  the nonfree locus 
$X \setminus X_{\fr}$ 
consists of finitely many closed points.
For any such point $x$, 
the subgroup scheme $H = H(x) \subset G$ 
is cyclic (Proposition \ref{prop:fibre}). 
Moreover, we have 
$X = \bigcup_{H,\nu} U(H,\nu)$, where 
$U(H,\nu)$ is the union of $X_{\fr}$ 
and those nonfree $x$ such that 
$H(x) = H$ and $G \cdot x$ has an equation 
of weight $\nu$ in $\cO_{X, G \cdot x}$
(see Theorem \ref{thm:ls} for these
results).  So Theorem \ref{thm:can} readily 
yields a version of the Hurwitz formula:

\begin{corollary}\label{cor:curves}
There is an isomorphism
of $G$-linearized sheaves
\begin{equation}\label{eqn:curves}
\omega_{X/Y} \simeq 
\cO_X \left(\sum \, (\vert H(x) \vert -1)  
\, G \cdot x \right) 
\end{equation}
(sum over the $G$-orbits of nonfree points),
which identifies $s_{X/Y}$ with the canonical
section of the right-hand side.
\end{corollary}

We now assume $X$ projective. Taking degrees in
(\ref{eqn:log}) and using the equality
$\pi^*(\pi(x)) = \vert H(x) \vert \, G \cdot x$
(which follows fromTheorem \ref{thm:ls}), 
we obtain 
\begin{equation}\label{eqn:hurwitz}
2 p_a(X) - 2 = 
\vert G \vert \, \big(2p_a(Y) - 2 
+ \sum_{y \in Y \setminus Y_{\fr}} 
( 1 - \frac{1}{n(y)} ) \, \deg(y) \big),
\end{equation}
where $n(y) = \vert H(x) \vert$
for any $y = \pi(x) \in Y \setminus Y_{\fr}$.  
This generalizes (\ref{eqn:hur}) to an arbitrary field.  

We may also determine the degree of the 
invertible sheaves $\cA_{\lambda}$
on $Y$, where $\lambda \in \Lambda$.
By Remark \ref{rem:secbis}, we have
the equality in $\Cl(Y)_{\bQ}$
\[ [\cA_{\lambda}] =
- \sum_{y \in Y \setminus Y_{\fr}}
\frac{m(y,\lambda)}{n(y)} y, \]
where $m(y,\lambda)$ denotes the unique integer 
such that $0 \leq m(y,\lambda) \leq n(y) -1$ 
and $\lambda - m(y,\lambda) \nu(y)$ restricts
trivially to $H(y) = H(x)$. 
As a consequence,
\begin{equation}\label{eqn:deg}
\deg(\cA_{\lambda}) = 
- \sum_{y \in Y \setminus Y_{\fr}}
\frac{m(y,\lambda)}{n(y)} \, \deg(y).
\end{equation}
Next, we determine the ``equivariant arithmetic 
genus'' of $X$, that is, the structure of 
the $G$-module $H^0(X,\omega_X)$. 
As $G$ is diagonalizable, 
it suffices to compute the character 
$\ch \, H^0(X,\omega_X)$, where the character 
of a finite-dimensional $G$-module 
$M = \bigoplus_{\lambda \in \Lambda} M_{\lambda}$ 
is the formal sum
$\ch(M) = \sum_{\lambda \in \Lambda} 
\dim(M_{\lambda}) e^{\lambda}$.

\begin{proposition}\label{prop:genus}
With the above notation, we have
\[ \ch \, H^0(X,\omega_X) - 1 =
\big( p_a(Y) - 1 \big) \, \ch \, \cO(G)
+ \sum_{y \in Y \setminus Y_{fr}}
\frac{\gamma(y)}{n(y)} \, 
\ch \, \cO\big( G/H(y) \big) \, \deg(y),\]
where 
$\gamma(y) = \sum_{\ell = 1}^{n(y) - 1} 
\ell e^{-\ell \nu(y)}$.
\end{proposition}

\begin{proof}
For any $\lambda \in \Lambda$, we set
$h^0(\omega_X)_{\lambda} = 
\dim \big (H^0(X,\omega_X)_{\lambda} \big)$.
If $G$ acts freely on $X$, then 
$\omega_X = \pi^*(\omega_Y)$
by (\ref{eqn:log}), or directly by 
Lemma \ref{lem:cantors}. As a consequence,
we obtain isomorphisms of $G$-modules
\[ H^0(X,\omega_X) \simeq 
H^0(Y,\omega_Y \otimes \cA),
\quad 
H^1(X,\omega_X) \simeq 
H^1(Y,\omega_Y \otimes \cA) \]
by using the projection formula. Moreover,
$H^1(X,\omega_X)$ is the trivial $G$-module
$k$ by Serre duality. This yields
$h^0(\omega_X)_0 = h^0(\omega_Y) = g(Y)$
and 
$h^0(\omega_X)_{\lambda} = 
\chi(Y, \omega_Y \otimes \cA_{\lambda})$
for any nonzero $\lambda$. Since 
$\deg(\cA_{\lambda}) = 0$ (as follows 
e.g.~from Remark \ref{rem:torsor}), 
we obtain
$h^0(\omega_X)_{\lambda} = g(Y) -1$
by using the Riemann--Roch theorem.
Thus,
$\ch \, H^0(X,\omega_X) - 1 =
\big( p_a(Y) - 1 \big) \, \ch \, \cO(G)$
as desired.

So we may assume that $G$ does not act 
freely on $X$, i.e., $G \cdot \Delta_X$ and 
$\Delta_Y$ are nonzero.
Denoting by $\iota$ the closed immersion
$G \cdot \Delta_X  \to X$,
we have a short exact sequence of 
$G$-linearized sheaves
\[ 0 \longrightarrow \omega_X
\longrightarrow \omega_X(G \cdot \Delta_X)
\longrightarrow 
\iota_*\iota^*\omega_X(G \cdot \Delta_X)
\longrightarrow 0
\]
and hence a long exact sequence of $G$-modules
\[ 0 \longrightarrow H^0(X,\omega_X)
\longrightarrow 
H^0 \big(X,\omega_X(G \cdot \Delta_X) \big)
\longrightarrow 
\bigoplus 
H^0 \big(G \cdot x,\omega_X(G \cdot \Delta_X) \big) 
\]
\[ \longrightarrow H^1(X,\omega_X)
\longrightarrow 
H^1 \big(X,\omega_X(G \cdot \Delta_X) \big)
\longrightarrow 0,
\]
where the sum runs over the $G$-orbits of
nonfree points. Moreover, 
\[ H^1 \big(X,\omega_X(G \cdot \Delta_X) \big) 
= H^0 \big(X, \cO_X(-G \cdot \Delta_X) \big)^* = 0. \]
In addition, we obtain isomorphisms
\[ H^0 \big(X,\omega_X(G \cdot \Delta_X) \big)_{\lambda} 
\simeq H^0 \big(Y,\omega_Y(\Delta_Y) 
\otimes \cA_{\lambda} \big) \]
for all $\lambda$, by arguing as above.
Thus,
\[
h^0 \big(\omega_X(G \cdot \Delta_X) \big)_{\lambda} 
= p_a(Y) - 1 + 
\sum_{y \in Y \setminus Y_{\fr}}
\big( 1 - \frac{m(y,\lambda)}{n(y)} \big) 
\, \deg(y) \]
in view of (\ref{eqn:deg}) and the 
Riemann-Roch theorem again. 

We also obtain isomorphisms of $G$-modules
$H^0 \big(G \cdot x, \omega_X(G \cdot \Delta_X) \big)
\simeq \cO(G \cdot x)$
by using (\ref{eqn:log}) and the fact
that $G \cdot x$ is contained
in the fibre of $\pi$ at $y = \pi(x)$. 
Moreover,
\[ \dim \cO(G \cdot x)_{\lambda} = 
\begin{cases}
\deg(y) & \text{if $m(y,\lambda) = 0$,}\\
0 & \text{else},
\end{cases}
\]
as follows from Proposition 
\ref{prop:fibre}. In particular, 
$h^0 \big(\omega_X(G \cdot \Delta_X) \big)_0 = 
p_a(Y) -1 + \deg(\Delta_Y)$ and 
$\sum \dim \cO(G \cdot x)_0 
= \deg(\Delta_Y)$. In view of our long 
exact sequence of $G$-modules, this 
yields the equality 
\[ h^0(\omega_X)_0 = p_a(Y). \] 
For $\lambda \neq 0$, we obtain similarly
\[ h^0(\omega_X)_{\lambda} = p_a(Y) - 1 
+ \sum
\big( 1 - \frac{m(y,\lambda)}{n(y)} \big) 
\, \deg(y), \]
where the sum runs over the points 
$y \in Y \setminus Y_{\fr}$ such that 
$m(y,\lambda) \neq 0$.
Translating the two latter equalities in 
terms of characters yields the assertion.
\end{proof}

The above proposition expresses the 
equivariant arithmetic genus as the sum
of a ``main term'' associated with a free
action, and of ``correcting terms'' 
associated with the nonfree points. 
Of course, it gives back the Hurwitz
formula (\ref{eqn:hurwitz}) by specializing 
characters to dimensions. But the latter 
formula has a much more direct proof.

Next, we obtain detailed information on
the local structure of the $G$-normal curve
$X$ in the case where $G$ is infinitesimal;
equivalently,  we have
$G \simeq \bfmu_{p^r}$ by 
\cite[Lem.~5.3]{Br24}.
(The general case reduces to this in view
of the factorization (\ref{eqn:quotients})
$X \stackrel{\varphi}{\longrightarrow}
Z = X/G^0 \stackrel{\psi}{\longrightarrow} Y$,
where $\psi$ is a finite abelian cover). 
We identify $G$ with $\bfmu_{p^r}$,
and hence the character group $\Lambda$
with $\bZ/p^r \bZ$. So the subgroup schemes 
of $G$ are exactly the $\bfmu_{p^s}$ for 
$s = 0,\ldots,r$.

Choose a nonfree point $x \in X$ and let
$A = \cO_{X,x}$,  $y = \pi(x)$ and 
$B = \cO_{Y,y}$ as in Section \ref{sec:lcs}; 
then $H(x) = \bfmu_{p^s}$ for 
some positive integer $s = s(x) \leq r$. 
We may now state a refinement of 
Proposition \ref{prop:local}:

\begin{proposition}\label{prop:localcurves}
With the above notation, we have an isomorphism
of $\bZ/p^r \bZ$-graded $B$-algebras
\[ 
A \simeq B[T_1,T_2]/
(T_1^{p^{r - s}} - u, 
T_2^{p^s} - t \, T_1^{\nu})
\]
where $T_1$ (resp.~$T_2$) is homogeneous
of weight $p^s$ (resp.~$\nu$ prime to $p$), 
$u \in B$ is a unit, and $t \in B$ generates
the maximal ideal.
\end{proposition}

\begin{proof}
By Corollary \ref{cor:factorization}, 
the morphism $X/H \to Y$
is a torsor under $G/H = \bfmu_{p^{r-s}}$
in a neighborhood of $y$. Together with
Remark \ref{rem:torsor}, this yields 
an isomorphism of $\bZ/p^r\bZ$-graded 
$B$-algebras
\[ A^H \simeq B[T_1]/
(T_1^{p^{r - s}} - u), \]
where $T_1$ is homogeneous of weight
$p^s$ and $u \in B^{\times}$.
We also have an isomorphism of 
$\bZ/p^r \bZ$-graded $B$-algebras
\[ A \simeq A^H[T_2]/(T_2^{p^s} - g), \]
where $T_2$ is homogeneous of weight
$\nu$ prime to $p$, and $g \in A^H$
is homogeneous of weight $p^s \nu$;
moreover, $g A = \fn A$ (Proposition
\ref{prop:local}). Thus, 
$g =  t \, T_1^{\nu}$, where $t \in B$ 
generates $\fn$. Combining both displayed 
isomorphisms yields the statement.
\end{proof}

\begin{remark}\label{rem:localcurves}
The above proposition takes a much simpler 
form in the case where $r = s$,  i.e., 
$x$ is a $G$-fixed point: then 
$A \simeq B[T]/(T^{p^r} - t)$
for a uniformizer $t \in B$.  As a consequence,
the homomorphism $B/Bt \to A/A f$
is bijective,  where $f$ denotes the image 
of $T$ in $A$.  
Equivalently, $X$ is regular at $x$ 
and $\kappa(x) = B/Bt = \kappa(y)$. This yields 
a refinement of \cite[Prop.~5.5]{Br24}.

On the other hand, if $r > s$ then we have
with the notation of Section \ref{sec:lcs}:
\[ 
\bar{A} = A/tA \simeq 
\kappa(y)[\bar{T_1},\bar{T_2}]/
(\bar{T_1}^{p^{r-s}} - \bar{u},\bar{T_2}^{p^s}),
\]
where $\bar{u} \in \kappa(y)^{\times}$ denotes 
the image of $u \in A^{\times}$.
If in addition $k$ is perfect, then so is
$\kappa(y)$ (as the extension $\kappa(y)/k$ is finite),  
and hence we obtain
\[ \bar{A}   \simeq 
\kappa(y)[U_1,U_2]/(U_1^{p^{r-s}},U_2^{p^s}).  \]
Thus,  $U_1, U_2$ form a minimal system of
generators of the ideal $\bar{\fm}$.  So 
$X$ is not regular at $x$, and $\kappa(x) = \kappa(y)$.
But if $\kappa(y)$ is imperfect, then we may choose
$u \in A^{\times}$ such that 
$\bar{u} \notin \kappa(y)^p$, and hence
$A/A f \simeq \kappa(y)[T]/(T^{p^{r-s}} - \bar{u})$
is a field, where $f$ denotes the image of 
$T_2$ in $A$.  In that case, $X$ is regular
but not smooth at $x$, and 
$\kappa(y) \subsetneq \kappa(x)$; the morphism
$\Spec \big(\kappa(x) \big) \to 
\Spec \big(\kappa(y) \big)$ 
is the torsor under $(\bfmu_{p^{r-s}})_{\kappa(y)}$ 
considered in Proposition \ref{prop:fibre}.
\end{remark}

Next, let $s$, $\nu$ be integers such that
$1 \leq s \leq r$ and 
$0 \leq \nu \leq p^r -1$, and let
$U = U(s,\nu)$ be the open subset of $X$
consisting of the free locus $X_{\fr}$
together with the closed points $x \in X$ 
such that $H(x) = \bfmu_{p^s}$ and 
the largest $\bZ/p^r \bZ$-graded ideal
of $\cO_{X,x}$ is generated by a homogeneous
element of weight $\nu$.
Then $V= V(s,\nu) = U/G$ is an open subset
of $Y$ containing $Y_{\fr}$, and  
$\Delta = \Delta(s,\nu) = V \setminus V_{\fr}$
is a reduced effective divisor on $V$.
With the notation of the beginning of 
Section \ref{sec:lcsc}, we have $U = U(H,\nu)$ 
and $\Delta = \Delta(H,\nu)$, where 
$H = \bfmu_{p^s}$. 
So the multiplication in $\cA \vert_V$ yields 
isomorphisms of $\cB \vert_V$-modules
\[
\sigma = \sigma_{\nu,p^s} : 
\cA_{\nu}^{ \otimes p^s} \vert_V
\stackrel{\sim}{\longrightarrow}
\cA_{\nu p^s}\vert_V (-\Delta),
\quad
\tau = \sigma_{p^s \nu,p^{r-s}} : 
\cA_{p^s}^{\otimes p^{r - s}} \vert_V
\stackrel{\sim}{\longrightarrow}
\cB \vert_V,
\]
since $\cA_{p^s} \subset \cA^H$ and the morphism
$U/H = \Spec_{\cB \vert_V}(\cA^H) \to V$ 
is a $G/H$-torsor (Theorem \ref{thm:ls}). 
We may now state:

\begin{proposition}\label{prop:curves}
There is an isomorphism of 
$\bZ/p^r \bZ$-graded $\cB \vert_V$-algebras
\[
\cA \vert_V \simeq 
\Sym_{\cB \vert_V}(\cA_{\nu} \oplus \cA_{p^s})/\cI,
\]
where $\cI$ denotes the ideal generated by the 
$a^{p^s} - \sigma(a^{\otimes p^s})$ 
where $a \in \cA_{\nu}$, and the
$b^{p^{r-s}} - \tau(b^{\otimes p^{r-s}})$ 
where $b \in \cA_{p^s}$.
\end{proposition}

\begin{proof}
Using the fact that 
$\Spec_{\cB \vert_V}(\cA^H) \to V$ 
is a $G/H$-torsor together with Remark
\ref{rem:torsor},  we obtain an isomorphism
\[ 
\cA^H \vert_V \simeq  
\Sym_{\cB \vert_V}(\cA_{p^s})/\cJ,
\]
where $\cJ$ denotes the ideal generated
by the 
$b^{p^{r-s}} - \tau(b^{\otimes p^{r-s}})$ 
for $b \in \cA_{p^s}$.
The assertion follows by combining 
this isomorphism with that of Theorem 
\ref{thm:ls} (ii).
\end{proof}

The above proposition takes again 
a much simpler form in the case where 
$r = s$: then we just have an isomorphism 
of $\cB \vert_V$-modules
$\sigma : \cA_{\nu}^{\otimes p^r}\vert_V
\stackrel{\sim}{\longrightarrow}
\cO_V(-\Delta)$,
and the algebra 
$\cA\vert_V $ is the quotient of 
$\Sym_{\cB \vert_V}(\cA_{\nu})$
by the ideal generated by the 
$a^{p^r} - \sigma(a^{\otimes p^r})$.
Moreover, $U$ is regular along $U^H$ 
by Remark \ref{rem:localcurves}.

Finally, we consider the tangent sheaf 
$\cT_X$ consisting of the $k$-linear 
derivations of the structure sheaf 
$\cO_X$, and its relative version
$\cT_{X/Y}$ consisting of the  
$\cO_Y$-linear derivations. Recall that
$\cT_X = \cHom_{\cO_X}(\Omega^1_X,\cO_X)$
is equipped with a $G$-linearization,
since so is $\Omega^1_X$; likewise,
$\cT_{X/Y}$ is $G$-linearized as well.
As a consequence, for any $x \in X$,
the fibre $\cT_X(x)$ is a linear 
representation of $H(x)$.
By \cite[Prop.~5.4]{Br24}, the sheaf
$\cT_X$ is invertible; in particular,
its fibre at $x$ is a vector space of 
dimension $1$ over the residue field 
$\kappa(x)$, for any $x$ as above. 
We will recover the latter result in 
a more explicit way:

\begin{corollary}\label{cor:tangent}
With the above notation, we have 
$\cT_U = \cT_{U/V}$. Moreover, $\cT_U(x)$ 
is a $\kappa(x)$-vector space of dimension $1$ 
on which $H(x)$ acts with weight $-\nu$, 
for any nonfree point $x \in U$.
\end{corollary}

\begin{proof}
Since $U$ is affine, the fraction field
of $\cO(U)$ is the function field $k(X)$. 
Moreover, the invariant subfield
$k(X)^G$ is generated by $k$ and 
$k(X)^{p^r}$ (see 
e.g.~\cite[Prop.~5.1]{Br24}). Thus, 
every $k$-linear derivation of $\cO(U)$
vanishes on $\cO(U)^G = \cO(V)$.
This shows the first assertion.

For the second assertion, we use the 
structure of the $\bZ/p^r\bZ$-graded
$B$-algebra $A$ (Proposition 
\ref{prop:curves}), which gives
an isomorphism of graded $A$-modules
\[ \Omega^1_{A/B} \simeq
(A \, dT_1 \oplus A \, dT_2)
/(A \, t \, T_1^{\nu-1} dT_1)
\]
with the notation of this proposition.
Taking $A$-duals, we obtain an 
isomorphism of graded $A$-modules
$\Der_B(A) \simeq A \, \partial_2$,
where $\partial_2$ arises from the
derivation $\partial/\partial T_2$
of $B[T_1,T_2]$. As $T_2$ has weight
$\nu$, this yields the result.
\end{proof}

For instance, if $x \in X(k)$ is
$G$-fixed, then it is a smooth point
of $X$ (e.g.~by Remark 
\ref{rem:localcurves}) and hence the 
Zariski tangent space $T_x(X)$ is a 
line on which $G$ acts with weight $-\nu$.

\begin{remark}\label{rem:tangent}
The degree of the tangent sheaf satisfies 
\[ 
\deg(\cT_X) = \vert G \vert
\sum_{y \in Y \setminus Y_{\fr}}
\frac{\deg(y)}{n(y)} = 
\sum_{y \in Y \setminus Y_{\fr}}
p^{r - s(y)} \, \deg(y). \] 
Indeed, $\cT_X$ is equipped with
a $G$-invariant global section: the generator
$\xi$ of the Lie algebra 
$\Lie(G) = \Lie(\bfmu_{p^r}) = \Lie(\bG_m) = k$.
Clearly, the divisor of $\xi$ is 
supported at the nonfree points. 
Given such a point $x$, we have
$\xi(T_1) = 0$ and $\xi(T_2) = \nu T_2$
with the notation of Proposition 
\ref{prop:localcurves}; thus,
$\xi = \nu \partial_2$. 
As a consequence, the ideal of zeroes of $\xi$
in $A$ equals $A T_2$. Since
\[ A/A T_2 = 
B[T_1]/(T_1^{p^{r-s}} -u, t T_1^{\nu}) 
= B[T_1]/(T_1^{p^{r-s}} -u, t) 
= \kappa(y)[T_1]/(T_1^{p^{r-s}} -u(y)), \]
we obtain
$\dim_{\kappa(y)}(A/A T_2) = p^{r - s} \deg(y)$.
This implies our formula.

In particular,  we have 
$\deg(\cT_X) \geq 0$, and equality holds
if and only if $X$ is a $G$-torsor over $Y$.
Also, note that the fixed point subscheme
$X^G$ is reduced and satisfies
$\deg(\cT_X) \equiv \deg(X^G)$ (mod $p$).
\end{remark}

\section{Linearly reductive group schemes}
\label{sec:lrgs}

Throughout this section, we denote by
$G$ a finite linearly reductive group
scheme, and by $X$ a $G$-normal variety
on which $G$ acts faithfully with quotient
$\pi : X \to Y$.  Our aim is to determine the 
relative dualizing sheaf $\omega_{\pi}$
as a $G$-linearized sheaf.  Using 
Lemmas \ref{lem:canS2} and \ref{lem:flatlci},
we may assume that $Y$ is regular
and $\pi$ is flat l.~c.~i.  Then 
$\omega_{\pi} = \omega_{X/Y}$
is invertible and has a canonical $G$-invariant 
section $s_{X/Y}$,  which trivializes it over 
the free locus and whose formation commutes
with flat base change on $Y$.

\begin{theorem}\label{thm:linred}
With the above assumptions,  we have
$\div(s_{X/Y}) = \pi^*(\Delta_Y) - G \cdot \Delta_X$.
\end{theorem}

\begin{proof}
We start with a reduction to the case where 
\emph{$G^0$ is diagonalizable and $\pi_0(G)$
is constant}.
As in Remark \ref{rem:linred},  we may choose
a finite Galois extension $K/k$ such that
$G^0_K$ is diagonalizable and $\pi_0(G)_K$
is constant.  Then the $G_K$-variety $X_K$
satisfies our assumptions; moreover,
$\Delta_{Y_K} = (\Delta_Y)_K$
as the base change to $K$ preserves
the free locus and reducedness,  and
$(G \cdot \Delta_X)_K = G_K \cdot (\Delta_X)_K$
as the formation of the schematic image commutes
with flat base change.  By Galois descent, 
it thus suffices to prove the theorem for
the $G_K$-variety $X_K$.

If $G = G^0$ then the desired assertion
follows from Remark \ref{rem:can}.
We now prove this assertion when
$G = \pi_0(G)$,  i.e.,  \emph{$G$ is constant}.
Since $G$ is linearly reductive,  its order is
prime to $p$; also,  recall that $X$ is normal.

We will need a folklore result 
for which we could not find any reference.
Consider a point  $x \in X$ with centralizer 
$C_G(x)$.  Then the quotient morphism 
$\pi : X \to Y$ factors as 
\[ X \stackrel{\varphi}{\longrightarrow}
Z = X/C_G(x) \stackrel{\psi}{\longrightarrow}
Y = X/G, \quad x \longmapsto z \longmapsto y. \]
In turn,  $\psi$ factors as 
\[ Z \stackrel{\tau}{\longrightarrow}
W = X/N_G(x) \stackrel{\eta}{\longrightarrow}
Y, \quad 
z \longmapsto w \longmapsto y, \]
where the normalizer $N_G(x)$ is the largest
subgroup of $G$ that stabilizes $x$, 
and $\tau$ is the quotient by $N_G(x)/C_G(x)$.

\begin{lemma}\label{lem:const}
With the above notation,  we have:

\begin{enumerate}

\item[{\rm (i)}] $\kappa(z) = \kappa(x)$.

\item[{\rm (ii)}] $\tau$ is a torsor under 
$N_G(x)/C_G(x)$ in a neighborhood of $z$.

\item[{\rm (iii)}]
$\eta$ induces an isomorphism
of completed local rings
$\widehat{\cO}_{Y,y} \simeq
\widehat{\cO}_{W,w}$.

\item[{\rm (iv)}]
$\psi$ is \'etale at $z$ and $s_{\psi}$
is a trivializing section of $\omega_{\psi}$ 
at that point.

\item[{\rm (v)}]
If $x$ is the generic point of a prime divisor,  
then $C_G(x)$ is cyclic.

\end{enumerate}

\end{lemma}

\begin{proof}
(i) This follows from the isomorphism 
of local rings 
$\cO_{Z,z} \simeq \cO_{X,x}^{C_G(x)}$
in view of the linear reductivity of 
$C_G(x)$.

(ii) Likewise,  the linear reductivity of $N_G(x)$ 
and the isomorphism  
$\cO_{W,w} \simeq \cO_{X,x}^{N_G(x)}$
imply that 
$\kappa(w) = \kappa(x)^{N_G(x)}
= \kappa(x)^{N_G(x)/C_G(x)}$.
As $N_G(x)/C_G(x)$ acts faithfully on $x$, 
the desired assertion follows from
Galois theory.

(iii) This is a special case of 
\cite[Exp.~V, Prop.~2.2]{SGA1})
and can be proved directly as follows:
$\pi$ induces an isomorphism
$\widehat{\cO}_{Y,y} \simeq
\widehat{\cO}_{X,G \cdot x}^G$. Moreover,
$G \cdot x$ is the disjoint union of the
$g \cdot x$, where $g \in G/N_G(x)$.
Thus, we have a $G$-equivariant isomorphism 
of semi-local rings
$\widehat{\cO}_{X,G \cdot x} \simeq
\prod_{g \in G/N_G(x)}
\widehat{\cO}_{X,g \cdot x}$.
Taking $G$-invariants, we obtain
$\widehat{\cO}_{Y,y} \simeq
\widehat{\cO}_{X,x}^{N_G(x)}$.
Moreover, 
$\cO_{X,x}^{N_G(x)} \simeq \cO_{W,w}$;
this yields the assertion.

(iv)The morphism $\tau$ is \'etale at $z$
by (ii),  and hence $\omega_{\tau}$
is invertible at that point; moreover, 
$s_{\tau}$ is a trivializing section 
in view of Lemma \ref{lem:cantors}. 
Also, $\eta$ is \'etale at $w$ by (iii),  so that
$\omega_{\eta}$ is invertible at that point. 
Moreover,  $s_{\eta}$ is a trivializing section, 
since its formation commutes with flat base 
change.  This implies the assertions in view 
of the isomorphism 
$\omega_{\psi} \simeq 
\omega_{\tau} \otimes \tau^*(\omega_{\eta})$
identifying $s_{\psi}$ with
$s_{\tau} \otimes \tau^*(s_{\eta})$.

(v) The group $C_G(x)$ acts faithfully
on the local ring $\cO_{X,x}$ and stabilizes 
the powers $\fm^n_x$ of the maximal ideal. 
Since $\bigcap_{n \geq 1} \, \fm^n_x = 0$,
the induced action on some quotient
$\cO_{X,x}/\fm^n_x$ is faithful as well.
By linear reductivity, 
it follows that the induced action on some
subquotient $\fm_x^m/\fm_x^{m+1}$
is faithful.  As $\fm_x^m/\fm_x^{m+1}$ is 
a vector space of dimension $1$ over 
$\kappa(x)$,  we obtain an injective 
homomorphism 
$C_G(x) \to \kappa(x)^{\times}$. 
\end{proof}

We may now prove Theorem \ref{thm:linred}
in the case where $G$ constant.  Let $D$ be 
a prime divisor in $X$,  with generic point $x$.  
By (\ref{eqn:adj}) and Lemma \ref{lem:const} (iv),
the multiplicity of $D$ in $s_{X/Y}$ equals 
that in $s_{X/Z}$.  The latter multiplicity equals
$\vert C_G(x) \vert -1$ in view of 
Theorem \ref{thm:can} together with
Lemma \ref{lem:const} (i), (v). 
On the other hand,  denoting by $E$
(resp.~$F$) the image of $D$ in $Y$ 
(resp.~$Z$),  we see that $F$ occurs
in $\psi^*(E)$ with multiplicity $1$
(as $\psi$ is \'etale at $z$),  and $D$ occurs
in $\varphi^*(F)$ with multiplicity
$\vert C_G(x) \vert$ (by Theorem
\ref{thm:ls}).  So $D$ occurs in 
$\pi^*(\Delta_Y) - G \cdot \Delta_X$ 
with multiplicity $\vert C_G(x) \vert -1$
as desired.

Finally, we handle the general case
(where $G^0$ is diagonalizable and $\pi_0(G)$ 
is constant of order prime to $p$).
Recall the factorization of $\pi$ as
\begin{equation}\label{eqn:fact}
X \stackrel{\varphi}{\longrightarrow}
Z = X/G^0 \stackrel{\psi}{\longrightarrow}
Y = X/G = Z/\pi_0(G).
\end{equation}
Moreover,  $X$ is a $G^0$-normal variety on 
which $G^0$ acts generically freely, 
and $Z$ is a normal variety on which 
$\pi_0(G)$ acts generically freely 
(Lemmas \ref{lem:fact} and \ref{lem:faithful}).
Thus,  $\psi$ is a tamely ramified Galois 
cover with group $\pi_0(G)(k)$.
We may further assume that $X$ is strongly
reductive relative to $G^0$.

Let  again $D$ be a prime divisor in $X$
with generic point $x$ and set $E = \pi(D)$,  
$F =\varphi(D)$ with generic points $y$,  $z$.
Then $F$ occurs in $\psi^*(E)$ with
multiplicity $\vert C_{\pi_0(G)}(z) \vert$
by the above step.  Moreover,  we have
\[ \varphi^*(F) = \vert H(x) \vert \, G^0 \cdot D \]
by Theorem \ref{thm:ls}.  Thus,
$\pi^*(\Delta_Y) - G \cdot  \Delta_X$
equals 
\[ \big( \vert H(x) \vert \, 
\vert C_{\pi_0(G)}(z) \vert -1 \big) \, G^0 \cdot D \]
in a neighborhood of $x$.  On the other hand,
we have 
\[ \div(s_{X/Y}) = \div(s_{X/Z}) + 
\varphi^* \div(s_{Y/Z}) \]
by (\ref{eqn:adj}),  and
\[ \div(s_{X/Z}) = \big( \vert H(x) \vert -1 \big)
G^0 \cdot D \]
in a neighborhood of $x$ (Theorem 
\ref{thm:can}),  whereas
\[ \div(s_{Y/Z}) =  
\big( \vert C_{\pi_0(G)}(z) \vert -1 \big) \, F \]
in a neighborhood of $z$ (by the above step
again).  As a consequence, $G^0 \cdot D$
occurs in  $\div(s_{X/Y})$ with multiplicity 
\[  \vert H(x) \vert -1 
+ \vert H(x) \vert \,  \big( \vert C_{\pi_0(G)}(z) \vert -1 \big)
= \vert H(x) \vert \,  \vert C_{\pi_0(G)}(z) \vert -1.  \]
So $\div(s_{X/Y})$ and 
$\pi^*(\Delta_Y) - G \cdot \Delta_X$ have the same
multiplicity along $G^0 \cdot D$.

To complete the proof,  it suffices to show
that $G^0 \cdot D$ coincides with $G \cdot D$
in a neighborhood of $x$. 
Since $\pi_0(G)$ is constant,  there exists
a finite purely inseparable extension $K/k$
such that 
$G_K = G^0_K \rtimes \pi_0(G)_K
= \coprod_{g \in G(K)} G^0_K \,  g$.
Therefore,  we have
$G_K \cdot D_K = \bigcup_{g \in G(K)}
G^0_K \cdot  g \cdot D_K$.
Moreover,  $D_K$ is irreducible
(possibly nonreduced).  Denoting by
$x'$ its generic point,  the generic points
of $G_K \cdot D_K$ are exactly the
$g \cdot x'$,  where $g \in G(K)$.
As a consequence,  $G^0_K \cdot D_K$
coincides with $G_K \cdot D_K$ in a
neighborhood of $x'$.  Taking the schematic
image under the projection $\pr: X_K \to X$,
$D_K \to D$,  $x' \to x$ yields the desired
assertion.
\end{proof}

Next,  we assume that $k$ is algebraically closed 
and $X$ is a curve. Then $Y = X/G$
is a smooth curve,  and $\pi : X \to Y$
is flat l.~c.~i.~by Lemma \ref{lem:flatlci}.
We now extend parts of Theorems
\ref{thm:local} and  \ref{thm:hurwitz} to this setting:

\begin{proposition}\label{prop:lr}

\begin{enumerate}
\item[{\rm (i)}]  The group $\Stab_G(x)$ is cyclic
for any $x \in X(k)$.

\item[{\rm (ii)}] We have
\[ 
\div(s_{X/Y}) = 
\sum \big( \vert \Stab_G(x) \vert \, -1 \big) 
\, G \cdot x 
\]
(sum over the $G$-orbits of $k$-rational 
points). 
\end{enumerate}

\end{proposition}

\begin{proof}
(i) Consider again the factorization (\ref{eqn:fact}) 
of $\pi$,  where $Z = X/G^0$  is a smooth curve.  
Also,  note that $G^0 \simeq \bfmu_{p^r}$ 
for some $r$.

Let $x \in X(k)$ and set 
$z = \varphi(x)$, $y = \psi(z) = \pi(x)$.
Since $k$ is algebraically closed,  we have
$G = G^0 \rtimes \pi_0(G)$ and 
$\Stab_G(x) = \Stab_{G^0}(x) \rtimes
\Stab_{\pi_0(G)}(x)$.  Moreover,
$\Stab_{G^0}(x) \simeq \bfmu_{p^s}$
for some $s$, and 
$\Stab_{\pi_0(G)}(x) = \Stab_{\pi_0(G)}(z)$ 
as $\varphi$ is bijective and 
$\pi_0(G)$-equivariant.  Also, 
$\Stab_{\pi_0(G)}(z)$ is cyclic
by Lemma \ref{lem:const} (v).

To complete the proof,  
it suffices to check that $\Stab_{\pi_0(G)}$
centralizes $\Stab_{G^0}(x)$,  since 
both groups are cyclic of coprime orders.
Consider the subgroup scheme
$H = G^0 \rtimes \Stab_{\pi_0(G)}(x)$
of $G$,  and the corresponding factorization
$X \to X/H \to Y$ of $\pi$.  Note that $X$ is 
an $H$-normal curve (since $H \supset G^0$); 
in particular, $X/H$ is a smooth curve.  Also, 
$\Stab_G(x) \subset H$.  Thus, we may replace
$G$ with $H$, and assume that 
$\pi_0(G)$ \emph{fixes} $x$. 

We now use some constructions and results
from Section \ref{sec:lcs}. Clearly, we have 
$G\cdot x = G^0 \cdot x$, and $G$ fixes $z$. 
Thus, $A = \cO_{X,G \cdot x} = \cO_{X,x}$
is a local algebra equipped with a $G$-action,
as well as its quotient algebra 
$\bar{A} = \cO(X_z)$. The ideal $\bar{I}$ of 
$G \cdot x$ in $\bar{A}$ is $G$-stable and
generated by a $G^0$-eigenvector $\bar{f}$ 
of weight $\nu \in \Lambda(G^0)$. The action 
of $\pi_0(G)$ on $G^0$ by conjugation yields
an action on $\Lambda(G^0)$ by group
automorphisms,  which stabilizes the subgroup 
$\Lambda \big(G^0/\Stab_{G^0}(x) \big)$
(as $\pi_0(G)$ normalizes $\Stab_{G^0}(x)$).
For any $g \in \pi_0(G)(k) = G(k)$,  we have 
$g \cdot \bar{f} = u_g \bar{f}$ for some
$u_g \in \bar{A}^{\times}$. Moreover,
$g \cdot \bar{f}$ is a $G^0$-eigenvector
of weight $g \cdot \nu$,  and hence 
$u_g$ is a $G^0$-eigenvector of weight 
$g \cdot \nu - \nu$.  By Proposition 
\ref{prop:fibre},  we have
$g \nu - \nu \in 
\Lambda \big(G^0/\Stab_{G^0}(x) \big)$
and the character group of 
$\Stab_{G^0}(x)$ is generated by the 
restriction of $\nu$.  As a consequence,
$\pi_0(G)$ acts trivially on this
character group.  Equivalently,
$\pi_0(G)$ centralizes $\Stab_{G^0}(x)$
as desired.

(ii) By the proof of Theorem \ref{thm:linred},
the orbit $G \cdot x$ (viewed as a Cartier divisor)
occurs in $\div(s_{X/Y})$ with multiplicity 
$\vert \Stab_{G^0}(x) \vert \,  
\vert \Stab_{\pi_0(G)}(z) \vert -1$.
Moreover,  
$\vert \Stab_{G^0}(x) \vert \,  
\vert \Stab_{\pi_0(G)}(z) \vert = 
\vert \Stab_{G^0}(x) \vert \,  
\vert \Stab_{\pi_0(G)}(x) \vert = 
\vert \Stab_G(x) \vert$.
\end{proof}

\begin{remark}\label{rem:fattening}
Still assuming $k$ algebraically closed,
we discuss the structure of finite linearly reductive 
group schemes $G$ acting faithfully on a curve.
Recall that  we have 
$G \simeq \bfmu_{p^r} \rtimes H$,
where $H$ is constant of order prime to $p$.  
As the automorphism group of 
$\bfmu_{p^r}$ is isomorphic to 
$(\bZ/p^r\bZ)^{\times} \simeq 
(\bZ/p\bZ)^{\times} \times \bZ/p^{r-1}\bZ$,
it follows that the conjugation action of $G$ on 
$\bfmu_{p^r}$ is given by a character 
\[ \chi : G \longrightarrow (\bZ/p\bZ)^{\times}, \]
which may be identified with the weight of the $G$-action 
on its Lie algebra.  In particular,  we have an exact 
sequence of finite group schemes
\begin{equation}\label{eqn:ext}
1 \longrightarrow \bfmu_{p^r} \times F
\longrightarrow G
\stackrel{\chi}{\longrightarrow}
(\bZ/p\bZ)^{\times}, 
\end{equation}
where $F$ (the centralizer of $\bfmu_{p^r}$ in $H$)
is constant of order prime to $p$.

Conversely,  given such an exact sequence,
does $G$ act faithfully on some curve (which may then 
be taken projective and $G$-normal)?
We only know a partial answer to this question, 
in the case where $G = \bfmu_{p^r} \times F$ with $F$ 
as above.  Then by the main result of \cite{MV},  
there exists a smooth projective curve $X$ such that 
$F = \Aut(X)$.   We first construct a faithful 
\emph{rational} action of $\bfmu_{p^r}$ on $X$ 
that  commutes with the $F$-action.

Let $K$ be the function field of $X$,
and $L = K^F$. Choose 
$t \in L \setminus L^p$; then 
$t \in K \setminus K^p$ and hence
we have a decomposition
\[ K = \bigoplus_{m = 0}^{p-1} 
K^p \, t^m, \]
which is a $\bZ/p\bZ$-grading of the
$K^p$-algebra $K$ by
$F$-stable subspaces. In view of
\cite[Cor.~3.4]{Br24}, this yields 
a faithful rational action of 
$\bfmu_p \times F$ on $X$, which extends 
the given $F$-action. More generally,
we obtain a decomposition
\[ K = \bigoplus_{m=0}^{p^r -1}
K^{p^r} \, t^m \]
by induction on $r$, and hence
the $F$-action on $X$ extends
to a faithful rational action of
$\bfmu_{p^r} \times F$. Moreover,
these actions are compatible with
the inclusions 
$\bfmu_{p^r}  \subset \bfmu_{p^{r+1}}$.

By loc.~cit., Cor.~4.4, these rational
actions can be ``regularized'', i.e.,  
for any $r$, there exists 
a $\bfmu_{p^r} \times F$-normal projective 
curve $X_r$ which is equivariantly 
birationally isomorphic to $X$; then
we still have $F = \Aut(X_r)$. 
If $g(X) \geq 2$ then $X_r \neq X$, 
as the automorphism group scheme $\Aut_X$
is \'etale. Under this assumption,
there exists no projective curve $Y$
that is $\bfmu_{p^r}$-normal for infinitely
many values of $r$, and birational to 
$X$: otherwise, the connected 
automorphism group scheme $\Aut^0_Y$ 
is not finite, and hence $Y$ must have
infinitely many automorphisms, 
a contradiction. In particular, 
we obtain infinitely many pairwise
nonisomorphic projective models $X_r$.
\end{remark}

\bibliographystyle{amsalpha}

%

\end{document}